\documentclass[10pt]{amsart}
\usepackage{amssymb,amsmath,txfonts,amsthm}
\usepackage{hyperref}
\usepackage{mathrsfs}
\newcommand\blfootnote[1]{%
  \begingroup
  \renewcommand\thefootnote{}\footnote{#1}%
  \addtocounter{footnote}{-1}%
  \endgroup
}
\newtheorem{theorem}{Theorem}
\newtheorem{prop}{Proposition}
\newtheorem{lemma}{Lemma}
\newtheorem{remark}{Remark}
\newtheorem{claim}{Claim}
\newtheorem{definition}{Definition}
\newtheorem{cor}{Corollary}

\numberwithin{equation}{section}
\numberwithin{theorem}{section}
\numberwithin{definition}{section}
\numberwithin{cor}{section}
\numberwithin{prop}{section}
\numberwithin{remark}{section}
\numberwithin{claim}{section}
\numberwithin{lemma}{section}

\def\Xint#1{\mathchoice
  {\XXint\displaystyle\textstyle{#1}}%
  {\XXint\textstyle\scriptstyle{#1}}%
  {\XXint\scriptstyle\scriptscriptstyle{#1}}%
  {\XXint\scriptscriptstyle\scriptscriptstyle{#1}}%
  \!\int}
\def\XXint#1#2#3{{\setbox0=\hbox{$#1{#2#3}{\int}$}
  \vcenter{\hbox{$#2#3$}}\kern-.5\wd0}}

\def\dashint{\Xint-}

\author{Gang Liu}
\address{Department of Mathematics\\University of California, Berkeley\\Berkeley, CA 94720}
\email{gangliu@math.berkeley.edu}
\title[Gromov-Hausdorff limit]{Gromov-Hausdorff limits of K\"ahler manifolds with bisectional curvature lower bound I}
\date{}
\begin{document}
\begin{abstract}
Given a sequence of complete(compact or noncompact) K\"ahler manifolds $M^n_i$ with bisectional curvature lower bound and noncollapsed volume, we prove that the pointed Gromov-Hausdorff limit is homeomorphic to a normal complex analytic space. The complex analytic structure is the natural ``limit" of complex structure of $M_i$. \end{abstract}\maketitle
\section{\bf{Introduction}}
\blfootnote{The author was partially supported by NSF grant DMS 1406593.}

In this paper, we consider the Gromov-Hausdorff limits of K\"ahler manifolds with bisectional curvature lower bound. The main interest is the degeneration of the complex structure. One motivation is from the uniformization conjecture of Yau which states that a complete noncompact K\"ahler manifold with positive bisectional curvature is biholomorphic to $\mathbb{C}^n$. Another motivation is from Alexandrov geometry or manifolds with sectional curvature lower bound, in particular, Perelman's stability theorem \cite{[P]}. For K\"ahler manifolds with bounded Ricci curvature or K\"ahler-Einstein case, see the notable works \cite{[DS]}\cite{[Ti5]}.
\begin{definition}\cite{[LW]} \cite{[TY]}
On a K\"ahler manifold $M^n$, we say the bisectional curvature is greater than or equal to $K$ (simply denoted by $BK\geq K$), if
\begin{equation}
\frac{R(X, \overline{X}, Y, \overline{Y})}{||X||^2||Y||^2+|\langle X, \overline{Y}\rangle|^2}\geq K 
\end{equation}
for any two nonzero vectors $X, Y\in T^{1, 0}M$.\end{definition} Observe that the equality holds for complex space forms.
Note that the bisectional curvature lower bound condition is weaker than the sectional curvature lower bound. It is stronger than the Ricci curvature lower bound. In fact, by taking the trace, we have $R_{i\overline{j}}\geq (n+1)Kg_{i\overline{j}}$.

\begin{theorem}\label{thm1}
Let $(M_\infty, p_\infty)$ be the pointed Gromov-Hausdorff limit of a sequence of complete(compact or noncompact) K\"ahler manifolds $(M^n_i, p_i)$ with $BK(M_i)\geq -1$ and $vol(B(p_i, 1))
\geq v>0$. Then $(M_\infty, p_\infty)$ is homeomorphic to a normal complex analytic space. \end{theorem}
\begin{remark}\label{rm1}
The complex analytic structure on $M_\infty$ is induced from the limit of holomorphic functions on small balls of $M_i$. Note this is very similar to \cite{[DS]}, where holomorphic functions are replaced by holomorphic sections. \end{remark}

\begin{remark}
The conclusion of theorem \ref{thm1} might be surprising at the first glance: the singularity of a normal complex analytic variety has real codimension at least $4$ while the metric singularity might have codimension $2$. To resolve this problem, we actually prove that metric singularities with tangent cones splitting off $\mathbb{R}^{2n-2}$ are regular in the complex analytic sense. Compare with \cite{[DS]}, where it was shown that complex analytic singularities are the same as metric singularities in the K\"ahler-Einstein case.
\end{remark}

It is a general fact that complex analytic spaces are locally contractible. See, for example, corollary $5.2$ in \cite{[D1]}. Therefore, we conclude the following
\begin{cor}
The limit space $M_\infty$ is locally contractible.
\end{cor}
\begin{remark} When the sectional curvature has a lower bound, the local contractibility of $M_\infty$ was proved in [30][32].
\end{remark}

During the proof of theorem \ref{thm1}, we obtain a topological result for complete K\"ahler surfaces with positive bisectional curvature:
\begin{cor}\label{cor0}
Let $(M^2, p)$ be a complete noncompact K\"ahler surface with positive bisectional curvature and maximal volume growth. Then $M$ is simply connected. Maximal volume growth means $vol(B(p, r))\geq cr^{4}$ for some $c>0$ and for all $r$.
\end{cor}
\begin{remark}
This result is rather weak. However, according to the author's knowledge, it is new. Indeed, there are very few results on topology of complete noncompact K\"ahler manifolds with positive bisectional curvature, even with the assumption that the manifold has maximal volume growth. In a forthcoming paper \cite{[L3]}, we shall continue to study the uniformization conjecture by using the results here.
\end{remark}

 Our strategy to theorem \ref{thm1} is an extension of techniques in \cite{[L2]} to the negatively curved case. We need the Gromov-Hausdorff convergence theory by Cheeger-Colding \cite{[CC1]}\cite{[CC2]}, Cheeger-Colding-Tian \cite{[CCT]}; adaptation of the heat flow theory by Ni-Tam \cite{[NT1]} to negatively curved case (note that here we essentially require the bisectional curvature lower bound, due to a Bochner formula of the complex hessian); H\"ormander's $L^2$-estimate \cite{[Ho]}\cite{[D]}; three circle theorem for negatively curved case \cite{[L1]}.
We also need to localize some argument in \cite{[DS]}.

This paper is organized as follows.
In section $2$, we collect some preliminary results.
Section $3$ is an extension of Ni-Tam's maximum principle to the negatively curved case. The proof is similar to the nonnegatively curved case \cite{[NT1]}.
In section $4$, we construct good holomorphic coordinates near special points of a K\"ahler manifold. Note this is crucial for that the complex analytic singularity has real codimension at least $4$. Section $5$ deals with the separation of points on the limit space. We construct holomorphic coordinates
on $M_\infty$ in section $6$. The structure sheaf on $M_\infty$ is introduced is section $7$. Finally, we complete the proof of theorem \ref{thm1} in section $8$.

Here are some conventions in this paper.
Let $e_\alpha$ be a local unitary frame of $T^{1, 0}M$ and $s$ be a smooth tensor on $M$.
Define $\Delta s = s_{\alpha\overline\alpha}+s_{\overline\alpha\alpha}$. Note this is twice the Laplacian defined in \cite{[NT1]}. Also define $|\nabla u|^2 = 2u_\alpha u_{\overline{\beta}}g^{\alpha\overline{\beta}}$. We will denote by $\Phi(u_1,..., u_k|....)$ any nonnegative functions depending on $u_1,..., u_k$ and some additional parameters such that when these parameters are fixed, $$\lim\limits_{u_k\to 0}\cdot\cdot\cdot\lim\limits_{u_1\to 0}\Phi(u_1,..., u_k|...) = 0.$$  Let $C(\cdot, \cdot, .., \cdot)$ and $c(\cdot, \cdot, .., \cdot)$ be large and small positive constants respectively, depending only on the parameters. The values might change from line to line. 

\medskip

\begin{center}
\bf  {\quad Acknowledgment}
\end{center}
The author would like to express his deep gratitude to Professors John Lott, Jiaping Wang for many valuable discussions during the work. He also thanks Professors Jeff Cheeger, Tobias Colding, William Minicozzi, Jian Song, Song Sun, Gang Tian for the interest in this work. He particularly thanks Professor Richard Bamler and Yuan Yuan for the careful reading and numerous suggestions.

\section{Preliminary results}
First recall some convergence results for manifolds with Ricci curvature lower bound. 
Let $(M^n_i, y_i, \rho_i)$ be a sequence of pointed complete Riemannian manifolds, where $y_i\in M^n_i$ and $\rho_i$ is the metric on $M^n_i$. By Gromov's compactness theorem, if $(M^n_i, y_i, \rho_i)$ have a uniform lower bound of the Ricci curvature, then a subsequence converges to some $(M_\infty, y_\infty, \rho_\infty)$ in the Gromov-Hausdorff topology. See \cite{[G]} for the definition and basic properties of Gromov-Hausdorff convergence.
\begin{definition}
Let $K_i\subset M^n_i\to K_\infty\subset M_\infty$ in the Gromov-Hausdorff topology. Assume $\{f_i\}_{i=1}^\infty$ are functions on $M^n_i$, $f_\infty$ is a function on $M_\infty$.  
$\Phi_i$ are $\epsilon_i$-Gromov-Hausdorff approximations, $\lim\limits_{i\to\infty} \epsilon_i = 0$. If $f_i\circ \Phi_i$ converges to $f_\infty$ uniformly, we say $f_i\to f_\infty$ uniformly over $K_i\to K_\infty$.
\end{definition}
 In many applications, $f_i$ are equicontinuous. The Arzela-Ascoli theorem applies to the case when the spaces are different.  When $(M_i^n, y_i, \rho_i)\to (M_\infty, y_\infty, \rho_\infty)$ in the Gromov-Hausdorff topology, any bounded, equicontinuous sequence of functions $f_i$ has a subsequence converging uniformly to some $f_\infty$ on $M_\infty$.

Let the complete pointed metric space $(M_\infty^n, y)$ be the Gromov-Hausdorff limit of a sequence of connected pointed Riemannian manifolds, $\{(M_i^n, p_i)\}$, with $Ric(M_i)\geq -(n-1)$ and $vol(B(p_i, 1))\geq v>0$. $M_\infty$ is called a noncollapsed limit. A tangent cone at $y\in M_\infty^m$ is a complete pointed Gromov-Hausdorff limit $((M_\infty)_y, d_\infty, y_\infty)$ of $\{(M_\infty, r_i^{-1}d, y)\}$, where $d, d_\infty$ are the metrics of $M_\infty, (M_\infty)_y$ respectively, $\{r_i\}$ is a positive sequence converging to $0$.
The following is theorem $5.2$ in \cite{[CC2]}:
\begin{theorem}\label{thm0}
Under the assumptions of the last paragraph, any tangent cone is a metric cone.
\end{theorem}
\begin{definition}\label{def1}
A point $y\in M_\infty$ is called $k$-weakly Euclidean, if some tangent cone splits off $\mathbb{R}^k$ isometrically. Let $\mathcal{W}\mathcal{E}_k$ denote the $k$-weakly Euclidean points. We also call $\mathcal{W}\mathcal{E}_n$ the set of regular points, denoted by $\mathcal{R}$. For any $\epsilon>0$, let $\mathcal{R}_{\epsilon}$ be the set of points $y\in M_\infty$ such that there exists $\delta>0$ with $d_{GH}(B(y, r), B_{\mathbb{R}^n}(0, r))<\epsilon r$ for all $0<r<\delta$.  Let $\mathring{\mathcal{R}}_\epsilon$ be the interior of $\mathcal{R}_{\epsilon}$.\end{definition}
In \cite{[CC2]}, the following theorem was proved: 
\begin{theorem}\label{thm-1}
The Hausdorff dimension of $M_\infty\backslash\mathcal{W}\mathcal{E}_k$ is at most $k-1$.
\end{theorem}

If in addition, $M_i$ are all K\"ahler, then theorem $9.1$ in \cite{[CCT]} states 
\begin{theorem}\label{thm-4}
$\mathcal{W}\mathcal{E}_{2k-1} = \mathcal{W}\mathcal{E}_{2k}.$
\end{theorem}

\medskip

H\"ormander's $L^2$ theory:
\begin{theorem}\label{thm-2}
Let $(X^n, \omega)$ be a connected but not necessarily complete K\"ahler manifold with $Ric\geq -(n+1)\epsilon(\epsilon>0)$. Assume $X$ is Stein. Let $\varphi$ be a $C^\infty$ function on $X$ with $\sqrt{-1}\partial\overline\partial\varphi \geq c\omega$ for some positive function $c>(n+1)\epsilon$ on $X$. Let $g$ be a smooth $(0, 1)$ form satisfying $\overline\partial g = 0$ and $\int_X\frac{|g|^2}{c-(n+1)\epsilon}e^{-\varphi}\omega^n<+\infty$, then there exists a smooth function $f$ on $X$ with $\overline\partial f = g$ and $\int_X |f|^2e^{-\varphi}\omega^n\leq \int_X\frac{|g|^2}{c-(n+1)\epsilon}e^{-\varphi}\omega^n$.\end{theorem}
The proof can be found in \cite{[D]}, page 38-39. Also compare with lemma 4.4.1 in \cite{[Ho]}. Note that the theorem also applies to singular metrics with positive curvature in the current sense.

\medskip
Three circle theorem in \cite{[L1]}:
\begin{theorem}\label{thm-3}
Let $M$ be a complete noncompact K\"ahler manifold with holomorphic sectional curvature $H\geq -1$, $p\in M$. Let $f$ be a holomorphic function on $M$. Let $M(r) = \sup\limits_{B(p, r)}|f(x)|$. Then $\log M(r)$ is a convex function of $\log\frac{e^r-1}{e^r+1}$. 
\end{theorem}

\section{A maximum principle for heat flow}
In this section we extend Ni-Tam's maximum principle \cite{[NT1]} to the negatively curved case.
 The proposition below is a modification of corollary $1.1$ in \cite{[NT1]}. 
 \begin{prop}\label{prop1}
Let $(M^n, p)$ be a complete noncompact K\"ahler manifold with $BK\geq -1$. Let $r(x) = d(x, p)$. Let $u$ be a nonnegative function on $M$ satisfying $u(x) \leq \exp(a+br(x))$ for some constants $a, b>0$. Let \begin{equation}\label{19}v(x, t) = \int_MH(x, y, t)u(y)dy.\end{equation} $H$ is the heat kernel on $M$. Then given any $1>\delta>0, T>0$, there exist $C_1>0, C_2>0$ depending only on $n, \delta, a, b, T$  such that for any $x\in M$ with $r=r(x)> C_2$, \begin{equation}\label{20}\frac{1}{2}\inf_{B(x, \delta r)}u\leq v(x, t)\leq C_1+\sup_{B(x, \delta r)}u\end{equation} for $0\leq t\leq T$.  The latter inequality holds for all $r$.
\end{prop}
\begin{remark}
The theorem also holds for compact manifolds. 
\end{remark}
\begin{proof}
Let $v = vol(B(p, 1))$.
Recall the heat kernel estimate \cite{[LY]}, there exists $C(n)>0$ with
\begin{equation}\label{21}H(x, y, t)\leq C(n)\frac{1}{\sqrt{vol(B(x, \sqrt{t}))vol(B(y, \sqrt{t}))}}\exp(-\frac{d^2(x, y)}{8t}+C(n)t).
\end{equation}

By volume comparison, \begin{equation}\label{22}vol(B(x, \sqrt{t}))\geq \frac{1}{C(n)}\exp(-8nr(x))v\min(t^{n}, 1),\end{equation}  \begin{equation}\label{23}vol(B(y, \sqrt{t}))\geq \frac{1}{C(n)}\exp(-8n(r(x)+d(x, y))v\min(t^{n}, 1), \end{equation}
\begin{equation}\label{24}\begin{aligned}\int_{M\backslash B(x, \delta r(x))}H(x, y, t)dy &\leq \frac{C(n)}{v\min(1, t^{n})}\int_{M\backslash B(x, \delta r(x))} \exp(8n(r(x)+d(x, y))-\frac{d^2(x, y)}{8t}+C(n)t)dy\\&\leq\frac{C(n, T)}{\min(1, t^{n})}\exp(80nr(x))\int_{\delta r(x)}^\infty\exp(16n\lambda-\frac{\lambda^2}{8t})d\lambda\\&\leq \frac{1}{2}\end{aligned}\end{equation} for $r(x)\geq C_2(n, T, \delta)$. 
As $u$ is of exponential growth, by (\ref{21}), we find that 
\begin{equation}\label{25}\int_{M\backslash B(x, \delta r(x))}H(x, y, t)u(y)dy\leq C_1(n, T, \delta, a, b).\end{equation}
Now \begin{equation}\label{26}\begin{aligned}v(x, t) &= \int_{B(x, \delta r(x))}H(x, y, t)u(y)dy + \int_{M\backslash B(x, \delta r(x))}H(x, y, t)u(y)dy\\&\leq \sup\limits_{B(x, \delta r)}u+C_1; \end{aligned}\end{equation}
\begin{equation}\label{27}\begin{aligned}v(x, t) &= \int_{B(x, \delta r(x))}H(x, y, t)u(y)dy + \int_{M\backslash B(x, \delta r(x))}H(x, y, t)u(y)dy\\&\geq (\inf\limits_{B(x, \delta r)}u)\int_{B(x, \delta r)}H(x, y, t)dy\\& \geq(1-\int_{M\backslash B(x, \delta r)}H(x, y, t)dy)\inf\limits_{B(x, \delta r)}u\\&\geq \frac{1}{2}\inf\limits_{B(x, \delta r)}u. \end{aligned}\end{equation}

\end{proof}

\begin{theorem}\label{thm2}
Let $(M^n, p)$ be a complete K\"ahler manifold. Let $r(x) = d(x, p)$. Assume the bisectional curvature is bounded from below by $-\epsilon$ for some $1>\epsilon>0$. Let $u$ be a smooth function on $M$ with compact support.  Let \begin{equation}\label{0}v(x, t) =\int_MH(x, y, t)u(y)dy.\end{equation} Here $H(x, y, t)$ is the heat kernel of $M$. Let $\eta(x, t)_{\alpha\overline\beta} = v_{\alpha\overline\beta}$ and $\lambda(x)$ be the minimal eigenvalue for $\eta(x, 0)-\epsilon|\nabla u(x)|^2g_{\alpha\overline\beta}$. Let \begin{equation}\label{1}\lambda(x, t) = \exp(8n\epsilon t)\int_M H(x, y, t)\lambda(y)dy.\end{equation} Then $\eta(x, t)-\lambda(x, t)g_{\alpha\overline{\beta}}-\epsilon|\nabla v(x, t)|^2g_{\alpha\overline{\beta}}+Ktg_{\alpha\overline{\beta}}$ is a nonnegative $(1, 1)$ tensor for $t\in [0, T]$, provided the following conditions are satisfied:
\begin{equation}\label{2}
8n\epsilon T<\frac{1}{2};
\end{equation}
\begin{equation}\label{3}
\frac{1}{2}K>8n\epsilon^2\exp(8n\epsilon T)\sup|\nabla u(x)|^2+8n^2\epsilon.
\end{equation}
\end{theorem}
\begin{remark}
We shall prove the theorem for the case when $M$ is noncompact. The proof for the compact case is even simpler.
\end{remark}
\begin{proof}
During the proof, $C, C_i(i\geq 1)$ will be large positive constants. The dependence will be clear from the context.
Following \cite{[NT1]}, we establish some bounds for $v$ and its derivatives.
\begin{lemma}\label{lm1}
\begin{equation}\label{4}(\frac{\partial}{\partial t}-\Delta)\eta_{\gamma\overline\delta} = 2R_{\beta\overline\alpha\gamma\overline\delta}\eta_{\alpha\overline\beta}-(R_{\gamma\overline{p}}\eta_{p\overline\delta}+R_{p\overline\delta}\eta_{\gamma\overline{p}}).
\end{equation}
For any $a>0$,
\begin{equation}\label{5}
\lim\limits_{r\to\infty}\inf\int_0^T\int_{B(p, r)}|\nabla v(x, t)|^2\exp(-ar^2(x))dxdt<\infty,
\end{equation}
\begin{equation}\label{6}
\lim\limits_{r\to\infty}\inf\int_0^T\int_{B(p, r)}||\eta||^2(x, t)\exp(-ar^2(x))dxdt<\infty.
\end{equation}
\end{lemma}
\begin{proof}
 (\ref{4}) follows from direct computation. 
 As $u$ has compact support, $|u|\leq C$. Then by the definition of $v$, $|v(x, t)|\leq C$ for all $x\in M, t\geq 0$.
Note
 \begin{equation}\label{7}(\Delta-\frac{\partial}{\partial t})v^2 = 2|\nabla v|^2.\end{equation}  We multiply (\ref{7}) by the standard cutoff function $\varphi^2$ supported in $B(p, 2r)$ with $\varphi = 1$ in $B(p, r)$ and $|\nabla\varphi|\leq \frac{5}{r}$. By integration by parts and volume comparison, we find \begin{equation}\label{8}\int_0^T\int_{B(p, r)}|\nabla v|^2\leq C_1(r^{-2}\int_0^{2T}\int_{B(p, 2r)}v^2+\int_{B(p, 2r)}u^2)\leq C_2(T+1)e^{50n\epsilon r}\end{equation} for $r\geq 1$. Then (\ref{5}) follows. 
 For the last equation,  we have 
 \begin{equation}\label{9}(\Delta-\frac{\partial}{\partial t})|\nabla v|^2 =4(|v_{ij}|^2+|v_{i\overline{j}}|^2+R_{i\overline{j}}v_jv_{\overline{i}})\geq 2|\nabla^2 v|^2-8n\epsilon|\nabla v|^2.\end{equation} By integration by parts as before, \begin{equation}\label{10}\int_0^T\int_{B(p, r)}|\nabla^2 v|^2\leq C_3((r^{-2}+8n\epsilon)\int_0^{2T}\int_{B(p, 2r)}|\nabla v|^2+\int_{B(p, 2r)}|\nabla u|^2)\leq C_4(T+1)e^{100n\epsilon r}\end{equation} for $r\geq 1$.
Then (\ref{6}) follows. 
 \end{proof}
 
Note (\ref{9}) implies that  \begin{equation}\label{11}(\Delta-\frac{\partial}{\partial t})(e^{-8n\epsilon t}|\nabla v(x, t)|^2) \geq 2e^{-8n\epsilon t}|\nabla^2 v|^2.\end{equation} 
Combining this with \begin{equation}\label{12}|\nabla|\nabla v||^2\leq |\nabla^2v|^2,\end{equation} we find
\begin{equation}\label{13}(\Delta-\frac{\partial}{\partial t})(e^{-4n\epsilon t}|\nabla v(x, t)|) \geq 0.\end{equation} 

 By the maximum principle in \cite{[LK]} or theorem $1.2$ in \cite{[NT2]}, (\ref{5}) and (\ref{13}), \begin{equation}\label{14}e^{-8n\epsilon t}|\nabla v(x, t)|^2\leq \max|\nabla u|^2.\end{equation} 
 At a point $x\in M$, we can diagonalize $\eta$ so that $\eta_{\alpha\overline\beta} = \lambda_\alpha\delta_{\alpha\beta}$. By direct calculations on page 477 of \cite{[NT1]}, 
\begin{equation}\label{15}\begin{aligned}
(\Delta-\frac{\partial}{\partial t})||\eta||^2& = 2|v_{\alpha\overline\beta s}|^2+2|v_{\alpha\overline\beta\overline s}|^2 +4R_{\alpha\overline p}v_{p\overline\delta}v_{\delta\alpha} -4R_{\alpha\overline\beta q\overline{p}}v_{p\overline{q}}v_{\beta\overline\alpha}\\& = 2|v_{\alpha\overline\beta s}|^2+2|v_{\alpha\overline\beta\overline s}|^2+2R_{\alpha\overline\alpha\beta\overline\beta}(\lambda_\alpha-\lambda_\beta)^2\\&\geq 2|v_{\alpha\overline\beta s}|^2+2|v_{\alpha\overline\beta\overline s}|^2-100\epsilon||\eta||^2\end{aligned}\end{equation} 
This implies that 
\begin{equation}\label{16}
(\Delta-\frac{\partial}{\partial t})(e^{-100\epsilon t}||\eta||^2)\geq (2|v_{\alpha\overline\beta s}|^2+2|v_{\alpha\overline\beta\overline s}|^2)e^{-100\epsilon t}.
\end{equation}
A direct calculation shows
\begin{equation}\label{17}
|\nabla ||\eta|||^2\leq |v_{\alpha\overline\beta s}|^2+|v_{\alpha\overline\beta\overline{s}}|^2.
\end{equation}
Then \begin{equation}\label{18}
(\Delta-\frac{\partial}{\partial t})(e^{-50\epsilon t}||\eta||)\geq 0.
\end{equation}
By (\ref{6}), we proved the following lemma:
  \begin{lemma}\label{lm2}
 $||\eta(y, t)||\leq e^{50\epsilon t}\max\limits_{x\in M}||\eta(x, 0)||$.
 \end{lemma} 
 
 Let $\phi(x) = \exp(r(x))$. Define \begin{equation}\label{28}\phi(x, t) = e^{8n\epsilon t}\int_MH(x, y, t)\phi(y)dy.\end{equation} Then \begin{equation}\label{29}(\frac{\partial}{\partial t}-\Delta)\phi = 8n\epsilon\phi\end{equation} and \begin{equation}\label{30}\phi(x, t)\geq ce^{c_1r}\end{equation} for $0\leq t\leq T$, by proposition \ref{prop1}. Here $c, c_1$ are positive constants. 
Given any $\tau>0$, consider \begin{equation}\label{31}(\tilde\eta)_{\alpha\overline\beta} = \eta(x, t)+(-\lambda(x, t)-\epsilon|\nabla v(x, t)|^2+Kt+\tau \phi(x, t))g_{\alpha\overline{\beta}}. \end{equation}
At $t=0$, $\tilde\eta> 0$. Also, for $0\leq t\leq T$, if $R$ is sufficiently large, by (\ref{14}), lemma \ref{lm2} and (\ref{30}), we have $\tilde\eta>0$ on $\partial B(p, R)$. Suppose at some $t_0\in [0, T]$, $\tilde\eta(x_0, t_0)<0$ for $x_0\in \overline{B(p, R)}$. Then there exists $0\leq t_1<T$ with $\tilde\eta(x, t)\geq 0$ for $x\in B(p, R)$ and $0\leq t\leq t_1$. Moreover, the minimum eigenvalue of $\tilde\eta(x_1, t_1)$ is zero for some $x_1\in B(p, R)$(note $x_1$ cannot be on the boundary).  Now we apply the maximal principle. Let us assume \begin{equation}\label{32}\tilde\eta(x_1, t_1)_{\gamma\overline{\gamma}}=0\end{equation} for $\gamma\in T^{1,0}_{x_1}M, |\gamma| =1$. We may diagonalize $\tilde\eta$ at $(x_1, t_1)$ and assume $\gamma$ is one of the basis of the holomorphic tangent space. Then at $(x_1, t_1)$,
 \begin{equation}\label{33}(\frac{\partial}{\partial t}-\Delta)\tilde\eta_{\gamma\overline{\gamma}}\leq 0.\end{equation} 
 
 On the other hand, by (\ref{4}), \begin{equation}\label{34}\begin{aligned}(\frac{\partial}{\partial t}-\Delta)\eta_{\gamma\overline{\gamma}} &= 2\sum\limits_{\alpha}R_{\gamma\overline{\gamma}\alpha\overline\alpha}\eta_{\alpha\overline\alpha}-2\sum\limits_{\alpha}R_{\gamma\overline{\gamma}\alpha\overline\alpha}\eta_{\gamma\overline{\gamma}}\\&=2\sum\limits_{\alpha}R_{\gamma\overline{\gamma}\alpha\overline\alpha}(\tilde\eta_{\alpha\overline\alpha}-\tilde\eta_{\gamma\overline\gamma})\\&\geq -2\epsilon\sum\limits_\alpha\tilde\eta_{\alpha\overline\alpha}\\&\geq-8n\epsilon(||\eta||-\lambda+Kt+\tau\phi).\end{aligned}\end{equation}

 Note by (\ref{1}), (\ref{9}) and (\ref{29}), \begin{equation}\label{35}\begin{aligned}&(\frac{\partial}{\partial t}-\Delta)(-\lambda(x, t)-\epsilon|\nabla v(x, t)|^2+Kt+\tau \phi(x, t))g_{\gamma\overline{\gamma}} \\&\geq -8n\epsilon\lambda+8n\epsilon\tau\phi+\epsilon(2||\eta||^2-8n\epsilon|\nabla v|^2)+K.\end{aligned}\end{equation} Hence at $(x_1, t_1)$, \begin{equation}\label{36}\begin{aligned}(\frac{\partial}{\partial t}-\Delta)\tilde\eta_{\gamma\overline{\gamma}}&\geq -8n\epsilon(||\eta||-\lambda+Kt+\tau\phi)+\\&-8n\epsilon\lambda+8n\epsilon\tau\phi+\epsilon(2||\eta||^2-8n\epsilon|\nabla v|^2)+K\\&\geq 2\epsilon(||\eta||-2n)^2-8n^2\epsilon-8n\epsilon^2|\nabla v|^2+(1-8n\epsilon t)K\\&>0,\end{aligned}\end{equation} according to (\ref{14}), (\ref{2}) and (\ref{3}). This contradicts (\ref{33}). The theorem follows if we first let $R\to\infty,$ then $\tau\to 0$.
 
\end{proof}
\begin{cor}\label{cor1}
Under the assumption of theorem \ref{thm2}, $\eta(x, t)_{\alpha\overline\beta}\geq (\lambda(x, t)-Kt)g_{\alpha\overline\beta}$.
\end{cor}

\section{Construct good holomorphic coordinates on manifolds}
In this section, we construct good holomorphic coordinates around certain points on a manifold. This will be crucial for that the complex analytic singularity has codimension at least $4$.

Let $0<\gamma\leq 2\pi$.
Let $(X, (0, o)) = (\mathbb{C}^{n-1}, 0)\times (Z, o)$ where $(Z, o)$ is a complex one dimensional cone with cone angle $\alpha$ satisfying $2\pi\geq \alpha\geq \gamma$. The metric on $(Z, o)$ is given by the standard metric $dr^2+r^2d\theta^2(0\leq\theta<\alpha)$.
On $X$, there is a global holomorphic chart $(z_1, ..., z_{n-1}, z_n)$: $z_1, ..., z_{n-1}$ are standard coordinates on $(\mathbb{C}^{n-1}, 0)$, $z_n(r, \theta) = r^\frac{2\pi}{\alpha}e^{\frac{2\pi\theta\sqrt{-1}}{\alpha}}$. It is clear that the coordinate functions are Lipschitz on each compact set of $X$. Let $K_r\subset\mathbb{C}^n$ be the image of $(z_1, ..., z_{n-1}, z_n)$ on $B_X((0, o), r)$. Then \begin{equation}\label{-1000}K_r = \{(z_1, ...., z_n)\in\mathbb{C}^n||z_1|^2+\cdot\cdot\cdot+|z_{n-1}|^2+|z_n|^{\frac{\alpha}{\pi}} \leq r^2\}\end{equation}Below is the main result in this section:
\begin{prop}\label{prop2}
Let $a>0$. There exist $\tilde\epsilon=\tilde\epsilon(n, \gamma)>0, \delta = \delta(n)>0$ so that the following hold. Assume $(M^n, x)$ is a complete K\"ahler manifold with $BK\geq-\frac{\epsilon^3}{a^2}$ for some $0<\epsilon<\tilde\epsilon$ and $d_{GH}(B(x, \frac{a}{\epsilon}), B_{X}((0, o), \frac{a}{\epsilon}))<\epsilon a$, then there exists a holomorphic chart $(w_1,...., w_n)$ containing $B(x, \delta a)$ so that 
\begin{itemize}
\item $w_s(x) = 0 (1\leq s\leq n)$.
\item  Up to an isometry of $(X, (0, o))$, on $B(x, \delta a)$, we have: for $1\leq i\leq n-1$, $w_i$ is $a\Phi(\epsilon|n, \gamma)$ close to $z_i$ under the Gromov-Hausdorff approximation; $w_{n}$ is $a^{\frac{2\pi}{\alpha}}\Phi(\epsilon|n, \gamma)$ close to $z_n$. In particular, on $B(x, \delta a)$, $|w_i|\leq C(n, \gamma)a (1\leq i\leq n-1)$ and $|w_n|\leq C(n, \gamma)a^{\frac{2\pi}{\alpha}}$.
\item The image of $(w_1, ..., w_n)$ contains the domain $K_{(\delta-\Phi(\epsilon|n, \gamma))a}$.\end{itemize} 
\end{prop}
\begin{proof}
It is clear that the proposition is independent of $a$. We may assume that $a$ is sufficiently large, to be determined. Let $a = 100R$.
Let $r(y)$ be the distance from $y$ to $x$.  We shall assume $\tilde\epsilon$ is sufficiently small. The value will be fixed later.
We first construct the weight function for H\"ormander's $L^2$ estimate. The argument follows from a slight modification of \cite{[L2]}. The completeness, we include most of the details. Set \begin{equation}\label{37}A= B(x, 5R)\backslash B(x, \frac{1}{5R}).\end{equation}  By the volume convergence theorem \cite{[C]} or theorem $5.9$ in \cite{[CC2]}, $A
$ satisfies the almost maximal volume condition (see $(4.8)$ or $(4.10)$ in \cite{[CC1]}).  By Cheeger-Colding theory \cite{[CC1]}((4.43) and (4.82)), there exists a smooth function $\rho$ on $M$ so that 
\begin{equation}\label{38}
\int_{A}|\nabla\rho-\nabla \frac{1}{2}r^2|^2 + |\nabla^2\rho-g|^2<\Phi(\epsilon|R, n, \gamma);
\end{equation}
\begin{equation}\label{39}
|\rho-\frac{r^2}{2}|<\Phi(\epsilon|R, n, \gamma)
 \end{equation} on $A$. 
 Let $F(r)$ be the Green function on $2n$ dimensional real space form with $Ric = -(n+1)\frac{\epsilon^3}{a^2}$. Then $F'(r)<0$. As $\epsilon\to 0$, up to a factor, 
 \begin{equation}\label{40}
 F\to r^{2-2n}, n>1; F\to \log r, n=1.
 \end{equation}
According to ($4.20$)-($4.23$) in \cite{[CC1]}, \begin{equation}\label{41}\rho =  \frac{1}{2}(F^{-1}\mathcal{G})^2; \Delta\mathcal{G}(y) = 0, y\in B(x, 10R)\backslash B(x, \frac{1}{10R});\end{equation}  \begin{equation}\label{42}\mathcal{G} = F(r)\end{equation} on $\partial (B(x, 10R)\backslash B(x, \frac{1}{10R})).$ Now \begin{equation}\label{43}|\nabla\rho(y)| = |F^{-1}\mathcal{G}||(F^{-1})'(\mathcal{G})||\nabla\mathcal{G}(y)|.\end{equation} By (\ref{39})-(\ref{42}) and Cheng-Yau's gradient estimate \cite{[CY]}, \begin{equation}\label{44}|\nabla\rho(y)|\leq C(n)r(y)\end{equation} for $y\in A$ and sufficiently small $\epsilon$ depending only on $n, R, \gamma$. Now consider a smooth function $\overline\varphi$: $\mathbb{R}^+\to \mathbb{R}^+$ given by $\overline\varphi(t) = t$ for $t\geq 2$; $\overline\varphi(t) = 0$ for $0\leq t\leq 1$; $|\overline\varphi|, |\overline\varphi'|, |\overline\varphi''|\leq C(n)$. Let \begin{equation}\label{45}u(y) = \frac{1}{R^2}\overline\varphi(R^2\rho(y)).\end{equation}  We set $u(y) = 0$ for $y\in B(x, \frac{1}{5R})$. Then $u$ is smooth on $B(x, 4R)$.
 \begin{claim}\label{cl1}
$\int_{B(x, 4R)}|\nabla u-\nabla \frac{1}{2}r^2|^2 + |\nabla^2u-g|^2<\Phi(\epsilon|R, n, \gamma); |u-\frac{r^2}{2}|<\Phi(\epsilon|R, n, \gamma)$ and $|\nabla u|\leq C(n)r$ on $B(x, 4R)$.
  \end{claim}
  \begin{proof}
  We have \begin{equation}\label{46}\nabla u(y) = \overline\varphi'(R^2\rho(y))\nabla\rho(y);\end{equation}\begin{equation}\label{47}\nabla^2 u(y) = R^2\overline\varphi''(R^2\rho(y))\nabla\rho\otimes\nabla\rho + \overline\varphi'(R^2\rho(y))\nabla^2\rho.\end{equation}
The proof follows from a routine calculation, by (\ref{38}), (\ref{39}) and (\ref{44}).
  \end{proof}
  
  Now consider a smooth function $\varphi$: $\mathbb{R}^+\to \mathbb{R}^+$ with $\varphi(t) = t$ for $0\leq t\leq 1$; $\varphi(t) = 0$ for $t\geq 2$; $|\varphi|, |\varphi'|, |\varphi''|\leq C(n)$. Let $H(z, y, t)$ be the heat kernel on $M$ and set \begin{equation}\label{48}h(y) = 5R^2\varphi(\frac{u(y)}{5R^2}), h_t(z) = \int_M H(z, y, t)h(y)dy.\end{equation}
\begin{claim}\label{cl2}
Assume $\epsilon$ is sufficiently small, depending only on $R, \gamma, n$. Then 
$(h_1)_{\alpha\overline\beta}(z)\geq c(n, \gamma)g_{\alpha\overline\beta}>0$ on $B(x, \frac{R}{10})$. 
\end{claim}  
\begin{proof}
Let $\lambda(y)$ be the lowest eigenvalue of $h_{\alpha\overline\beta}-\frac{\epsilon^3}{a^2} |\nabla h|^2g_{\alpha\overline\beta}$. Let \begin{equation}\label{49}\lambda(z, t) = e^{8n\frac{\epsilon^3}{a^2}}\int_M H(z, y, t)\lambda(y)dy.\end{equation}  By corollary \ref{cor1},  \begin{equation}\label{50}(h_1)_{\alpha\overline\beta}(z)\geq (\lambda(z, 1)-K)g_{\alpha\overline\beta}, \end{equation} provided the following inequalities are satisfied:
\begin{equation}\label{51}
8n\frac{\epsilon^3}{a^2}<\frac{1}{2},
\end{equation}
\begin{equation}\label{52}
\frac{1}{2}K>8n(\frac{\epsilon^3}{a^2})^2\exp(8n\frac{\epsilon^3}{a^2})\sup|\nabla h|^2+8n^2\frac{\epsilon^3}{a^2}.\end{equation}
From (\ref{48}) and claim \ref{cl1}, it is clear that $|\nabla h|\leq C(n)R$ on $M$. If $\epsilon$ is very small, we can make $K$ small and (\ref{51}), (\ref{52}) hold.  To prove claim \ref{cl2}, it suffices to prove that $\lambda(z, 1)\geq c(n, \gamma)>0$ for $z\in B(x, \frac{R}{10})$. The proof is almost the same as in claim $1$ in \cite{[L2]}. We skip the details here.
\end{proof}
\begin{claim}\label{cl3}
There exist $\epsilon_0=\epsilon_0(n)>0$, $R\geq C_0(n)>100$, $\epsilon = \epsilon(n, R, \gamma)$ sufficiently small so that \begin{equation}\label{eq18}\min\limits_{y\in\partial B(x, \frac{R}{20})}h_1(y)> 4\sup\limits_{y\in B(x, \epsilon_0R)}h_1(y).\end{equation} Also $0\leq h_1(y)\leq C(n, \gamma)R^2$ on $B(x, R)$.
\end{claim}
\begin{proof}
According to $(\ref{39})$, this is a consequence of proposition \ref{prop1}.
\end{proof}

Now we freeze the value $R = C_0(n)$ in claim \ref{cl3}.  Then $\epsilon$ depends only on $n$ and $\gamma$.
We might make $\epsilon$ even smaller later.
Let $\Omega$ be the connected component of $\{y\in B(x, \frac{R}{20})|h_1(y)<2\sup\limits_{y\in B(x, \epsilon_0R)}h_1(y)\}$ containing $B(x, \epsilon_0R)$. Then $\Omega$ is relatively compact in $B(x, \frac{R}{20})$ and $\Omega$ is a Stein manifold by claim \ref{cl2}.

\begin{lemma}\label{lm3}
There exist complex harmonic functions $w'_i(1\leq i\leq n)$ on $B(x, 2R)$ so that the following hold.
\begin{itemize}
\item Up to an isometry of $(X, (0, o))$, on $B(x, 2R)$, we have for $1\leq i\leq n$, $w'_i$ is $\Phi(\epsilon|n, \gamma)$ close to $z_i$ under the Gromov-Hausdorff approximation.
\item $\int_{B(x, R)}|\overline\partial w'_i|^2\leq \Phi(\epsilon|n, \gamma)$ for $1\leq i\leq n$.
\end{itemize}
\end{lemma}
\begin{proof}
First we construct $w'_i$ for $1\leq i\leq n-1$.  The construction is similar to proposition $1$ in \cite{[L2]}. For completeness, we include the details. According to Cheeger-Colding theory \cite{[CC1]}(also equation ($1.23$) in \cite{[CC3]}), 
there exist real harmonic functions $b_1,..., b_{2n-2}$ on $B(x,4R)$ so that 
\begin{equation}\label{56}\dashint_{B(x, 2R)}  \sum\limits_{s}|\nabla(\nabla b_s)|^2 +\sum\limits_{s, l}|\langle\nabla b_s, \nabla b_l\rangle - \delta_{sl}|^2\leq \Phi(\epsilon|n, \gamma)\end{equation} and \begin{equation}\label{57}b_s(x) = 0(1\leq s\leq 2n-2); |\nabla b_s|\leq C(n)\end{equation} on $B(x, 2R)$. Moreover, the map $F(y) = (b_1(y),..., b_{2n-2}(y))$ is a $\Phi(\epsilon|n, \gamma)$ approximation to the Euclidean factor of $X$.
According to the argument above lemma $9.14$ in \cite{[CCT]}(see also $(20)$ in \cite{[L]}), after a suitable orthogonal transformation, we may assume \begin{equation}\label{58}\dashint_{B(x, 2R)} |J\nabla b_{2j-1} - \nabla b_{2j}|^2 \leq \Phi(\epsilon|n, \gamma)\end{equation} for $1\leq j\leq n-1$. Set $w'_j = b_{2j-1}+\sqrt{-1}b_{2j}$. Then \begin{equation}\label{59}\dashint_{B(x, 2R)} |\overline\partial w'_j|^2\leq \Phi(\epsilon|n, \gamma).\end{equation} By composing with an isometry of $(X, (0, o))$, we may assume $w_j'$ is close to $z_j$.

Now we construct the function $w_n'$. It is clear that $z_n$ is Lipschitz on $\partial B((0, o), 2R)$. We can transplant it to $\partial B(x, 2R)$ as a Lipschitz function $h'_n$. Basically we first transplant the values to a $\delta$-net, then extend to a Lipschitz function by Macshane lemma (see, for example, (8.2) in \cite{[Ch1]}). One can also directly apply lemma $10.7$ in \cite{[Ch1]}. We may assume $h'_n$ is very close to $z_n|_{\partial B((0, o), 2R)}$.
Following Ding \cite{[Di2]}, we solve the Dirichlet problem $\Delta w_n' = 0$ with boundary data $w'_n = h'_n$. By using the same arguments as in theorem $2.1$ of \cite{[Di2]}(replace $b_i$ in $(2.3)$ of \cite{[Di2]} by the Green function on the space form with $Ric = -(n+1)\frac{\epsilon^3}{a^2}$), we find that $w'_n$ is close to $z_n$ up to $\Phi(\epsilon|n, \gamma)$ error. 

Next we prove that $w_n'$ is almost holomorphic or anti-holomorphic on $B(x, R)$. More precisely, we prove 
\begin{equation}\label{60}
\int_{B(x, R)}|D w'_n|^2\leq \Phi(\epsilon|n, \gamma)
\end{equation}
where $D = \partial$ or $\overline\partial$. 
We always assume $\epsilon$ is as small as we want. Let $S = \{y\in X|z_n(y) = 0\}$. That is, $S$ is the set of singular points of $X$.
Fix small $\epsilon'>0$. Given any point $y'\in B(x, \frac{3}{2}R)\backslash B(S, \epsilon' R)$($B(S, \epsilon'R)$ is just the distance neighborhood of $S$), we can find $y\in B_X((0, o), \frac{3}{2}R)$ with $y$ close to $y'$ up to distance $\Phi(\epsilon|n, \gamma)$. Since $X$ is flat outside $S$,  there exist $\delta''=\delta''(n, \epsilon', \gamma)$ and
a holomorphic chart $(a_1, ..., a_n)$ in $B(y, 2\delta'' R)$ with $a_i = z_i$ for $1\leq i\leq n-1$ and  the metric is given by $\omega = \frac{\sqrt{-1}}{2}\sum\limits_{i=1}^{n}da_i\wedge\overline{da_i}$. This means that each $a_i$ is a parallel coordinate function. Furthermore, we can require that $a_n$ is a function depending only on $z_n$. Thus we can regard $z_n$ as a function of $a_n$. 

As we mentioned before, if $\epsilon$ is sufficiently small, $B(y, 2\delta''R)$ is close to $B(y', 2\delta''R)$ as we want.
According to Cheeger-Colding theory \cite{[CC1]}, we can find $0<\delta' = \delta'(n, \gamma, \epsilon', \delta'')<<\delta''$ and complex harmonic functions $(a''_1, ..., a''_n)$ on $B(y', 2\delta'R)$ with $a''_i$ close to $a_i$ up to error $\epsilon'\delta'R$. Furthermore, 
\begin{equation}\label{eq0}
\dashint_{B(y', \delta' R)}\sum\limits_{1\leq i, j\leq n}(|\langle da''_i, \overline{da''_j} \rangle-2\delta_{ij}|^2+|\langle da''_i, da''_j \rangle|^2)<\epsilon'; |da''_i|\leq C(n).\end{equation}
By assume $\delta'$ be sufficiently small, we may assume \begin{equation}\label{eq1}|z_n(t) - \frac{\partial z_n}{\partial a_n}(y)a_n(t)-(z_n(y)-\frac{\partial z_n}{\partial a_n}(y)a_n(y))|<\epsilon'\delta'R\end{equation} for any $t\in B(y, 2\delta'R)$. This merely says $z_n$ is almost linear in terms of $a_n$ on $B(y, 2\delta'R)$. 
For notational convenience, we set $\lambda_1(y) = z_n(y)-\frac{\partial z_n}{\partial a_n}(y)a_n(y)$.
Recall the definition of $z_n$ in the second paragraph of this section. Since $z_n$ depends only on $a_n$ and the metric $\omega = \frac{\sqrt{-1}}{2}\sum\limits_{i=1}^nda_i\wedge\overline{da_i}$ on $B(y, 2\delta''R)$,  \begin{equation}\label{-1}|dz_n(y)| = |\frac{\partial z_n}{\partial a_n}(y)da_n| = \sqrt{2}\frac{2\pi}{\alpha}r(y)^{\frac{2\pi}{\alpha}-1}\leq C(n, \gamma).\end{equation} We have used that $R$ depends only on $n, \gamma$.

Let $a'_j(1\leq j\leq n-1)$ be the restriction of $w'_j$ on $B(y', 2\delta'R)$. Let $a'_n=a''_n$. By the sentence below (\ref{59}), we may assume $a'_j$ is close to $z_j(1\leq j\leq n-1)$ up to error $2\epsilon'\delta'R$. Since $a_j = z_j$ for $1\leq j\leq n-1$ on $B(y, \delta''R)$, by the sentence above (\ref{eq0}), we find that on $B(y, \delta'R)$,
\begin{equation}\label{eq2}
|a'_j-a''_j|\leq 10\epsilon'\delta'R
\end{equation}
As $a'_j$ and $a''_j$ are harmonic, gradient estimate says on $B(y', \delta'R)$, \begin{equation}\label{eq3}|da'_j-da''_j|\leq C(n, \gamma)\epsilon'.\end{equation}\begin{claim}\label{claim0}
  \begin{equation}\label{61}
\dashint_{B(y', \delta' R)}\sum\limits_{1\leq i, j\leq n}|\langle da'_i, \overline{da'_j} \rangle-2\delta_{ij}|^2<C(n, \gamma)\epsilon'; \dashint_{B(y', \delta' R)}|\overline\partial a'_j|^2<C(n, \gamma)\epsilon', 1\leq j\leq n-1;\end{equation} \begin{equation}\label{62}\dashint_{B(y', \delta' R)}\sum\limits_{1\leq i, j\leq n}|\langle da'_i, da'_j \rangle|^2<C(n, \gamma)\epsilon';\dashint_{B(y', \delta' R)}|D a'_n|^2<C(n, \gamma)\epsilon'\end{equation} for $D = \partial$ or $\overline\partial$. $D$ can only be one of them, if $\epsilon'$ is sufficiently small. \end{claim}
\begin{proof}
(\ref{61}) and the first inequality in (\ref{62}) follow from (\ref{59}), (\ref{eq3}), (\ref{eq0}). For the last inequality, one can apply the same argument as in (\ref{58}). If (\ref{62}) holds for $D=\partial$ and $D=\overline\partial$, $\dashint_{B(y', \delta' R)}|da'_n|^2\leq C(n, \gamma)\epsilon'$. This contradicts (\ref{61}).\end{proof}

\medskip

Recall the function $\lambda_1(y)$ defined below (\ref{eq1}). Set \begin{equation}\label{-2}\tilde{z}_n(s) = \frac{\partial z_n}{\partial a_n}(y)a'_n(s)+\lambda_1(y)\end{equation} for $s\in B(y', \frac{3}{2}\delta'R)$. Then $\tilde{z_n}$ is harmonic. By (\ref{-1}) and (\ref{62}),
\begin{equation}\label{63}
\dashint_{B(y', \delta' R)}|D \tilde{z}_n|^2<C(n, \gamma)\epsilon'.\end{equation}
\begin{claim}\label{cl4}
$|d\tilde{z}_n-dw'_n|<C(n, \gamma)\epsilon'$ on $B(y', \delta' R)$. Thus\begin{equation}\label{64}
\dashint_{B(y', \delta' R)}|d\tilde{z}_n-dw'_n|^2<C(n, \gamma)\epsilon'.\end{equation}
\end{claim}
\begin{proof}
As $a_n$ is close to $a'_n$ up to error $C(n)\epsilon'\delta'R$, by (\ref{eq1}) and (\ref{-2}), $z_n$ is close to $\tilde{z}_n$ up to error $C(n, \gamma)\epsilon'\delta'R$. By the paragraph above (\ref{60}), $z_n$ is also close to $w'_n$ up to error $\Phi(\epsilon|n, \gamma)$. We can make this as small as we want. Thus we may assume $|\tilde{z}_n - w'_n|\leq C(n, \gamma)\epsilon'\delta'R$ on $B(y', \frac{3}{2}\delta'R)$. Cheng-Yau's gradient estimate implies the desired claim.
\end{proof}

By claim \ref{cl4} and (\ref{63}), we find 
\begin{equation}\label{65}
\dashint_{B(y', \delta' R)}|Dw'_n|^2<C(n, \gamma)\epsilon'.\end{equation}
Let $S'\in B(x, \frac{3}{2}R)$ be the preimage of $S$ under the Gromov-Haudorff approximation. This is rough, but enough for purpose.
If $\epsilon'<<1$, the type of $D$ does not change when $y'$ moves in $B(x, \frac{3}{2}R)\backslash B(S', 2\epsilon' R)$. 
We can consider covering of $\overline{B(x, \frac{5}{4}R)}\backslash B(S', 2\epsilon' R)$ by balls $B(y', \delta'R)$ so that each point belongs only to at most $C(n, \gamma)$ balls.
This implies that 
\begin{equation}\label{66}
\dashint_{B(x, R)\backslash B(S', 2\epsilon' R)}|Dw'_n|^2<C(n, \gamma)\epsilon'.\end{equation}
Gradient estimate says 
$|dw'_n|\leq C(n, \gamma)$ on $B(x, R)$.
The volume convergence theorem \cite{[C]} says \begin{equation}\label{67}Vol(B(S', 2\epsilon'R)\cap B(x, R))\leq Vol(B(S, 3\epsilon'R)\cap B_X((0, o), R))+\Phi(\epsilon|n, \gamma).\end{equation}
Therefore, we have \begin{equation}\label{68}
\int_{B(S', 2\epsilon'R)\cap B(x, R)}|dw'_n|^2\leq C(n, \gamma)\epsilon'^{2}
\end{equation}
(\ref{66}) and (\ref{68}) imply 
\begin{equation}\label{69}
\int_{B(x, R)}|Dw'_n|^2\leq C(n, \gamma)\epsilon'.
\end{equation}
Given any $\epsilon'>0$, we can find small $\epsilon>0$ so that the inequalities above all hold.
By taking the conjugate of $w'_n$ if necessary, 
we conclude the proof of lemma \ref{lm3}.

 \end{proof}

Now we are ready to solve the $\overline\partial$-problem $\overline\partial f_i = \overline\partial w'_i$. By claim \ref{cl2}, claim \ref{cl3}, theorem \ref{thm-2}, lemma \ref{lm3} and the definition of $\Omega$ below claim \ref{cl3}, 
\begin{equation}\label{70}\int_{\Omega}|f_i|^2e^{-h_1}\leq \frac{1}{c(n, \gamma)}\int_{\Omega}|\overline\partial w'_i|^2e^{-h_1}<\Phi(\epsilon|\gamma, n).\end{equation} As $w_i'$ is harmonic, $f_i$ is harmonic. Therefore, mean value theorem \cite{[LS]} and gradient estimate imply that \begin{equation}\label{71}|f_i|,  |\nabla f_i|\leq\Phi(\epsilon|\gamma, n)\end{equation} on $B(x, \frac{1}{2}\epsilon_0R)$. Set $w_i = w'_i-f_i$. We can do a perturbation so that $w_i(x) = 0$.
Next we prove that $(w_1, ..., w_n)$ is a holomorphic chart on $B(x, \frac{\epsilon_0}{4}R)$. 
\begin{claim}\label{cl5}
$\int_{B(x, \frac{\epsilon_0R}{2})}\sum\limits_{1\leq i\leq n-1; 1\leq j\leq n}|\langle dw_i(y'), \overline{dw_j(y')}\rangle-2\delta_{ij}|^2+||dw_n(y')|-\sqrt{2}\frac{2\pi}{\alpha}r(y')^{\frac{2\pi}{\alpha}-1}|dy'<\Phi(\epsilon|n, \gamma)$.
\end{claim}
\begin{proof}
By the definition of $z_n$ right above proposition \ref{prop2} and (\ref{-1}), $|\frac{\partial z_n}{\partial a_n}(y)| = \frac{2\pi}{\alpha}r(y)^{\frac{2\pi}{\alpha}-1}$.
The proof follows from (\ref{56}), (\ref{61}), (\ref{-2}), (\ref{71}) and claim \ref{cl4}.
\end{proof}
Recall $K_r$ is defined in (\ref{-1000}). By claim \ref{cl5} and that $w_i$ are holomorphic, we have
\begin{claim}\label{cl6}
$|\frac{1}{(2\sqrt{-1})^n}\int_{B(x, \frac{\epsilon_0R}{2})}dw_1\wedge \overline{dw_1}\wedge\cdot\cdot\cdot\wedge dw_n\wedge\overline{dw_n} -vol(K_\frac{\epsilon_0R}{2})|\leq \Phi(\epsilon|n, \gamma)$.
\end{claim}
Set $w = (w_1, ..., w_n)$. By lemma \ref{lm3}, $w^{-1}(K_{\frac{\epsilon_0R}{2}-\Phi(\epsilon|n, \gamma)})$ is relatively compact in $B(x, \frac{\epsilon_0R}{2})$. Take the connected component $K'$ of $w^{-1}(K_{\frac{\epsilon_0R}{2}-\Phi(\epsilon|n, \gamma)})$ containing $B(x, \frac{\epsilon_0 R}{4})$. Then $w: K'\to K_{\frac{\epsilon_0R}{2}-\Phi(\epsilon|n, \gamma)}$ is proper.
Claim \ref{cl6} implies that if $\epsilon$ is sufficiently small, the degree of $w$ is $1$. Thus $w$ is generically one to one on $K'$. In particular, it is surjective. By the first conclusion of lemma \ref{lm3},  $w(B(x, \frac{1}{4}\epsilon_0R))$ contains $K_{\frac{1}{4}\epsilon_0R-\Phi(\epsilon|n, \gamma)}$.
Observe $w$ is a finite map, as the preimage of a point is a subvariety which is compact in the Stein manifold $\Omega$. According to proposition $14.7$ on page $87$ of \cite{[GPR]}, $w$ is an isomorphism on $B(x, \frac{1}{4}\epsilon_0R)$. Now we can find the values of $\tilde\epsilon$ and $\delta$ required in proposition \ref{prop2}. This concludes the proof.
\end{proof}

\begin{cor}\label{cor2}
Let $(Y^n, x)$ be a complete K\"ahler manifold with bisectional curvature bounded from below by $-1$ and $vol(B(x, 1))\geq v>0$. Then there exist $0<\epsilon'<<1$, $\delta_5, \delta_6<1$ depending only on $n, v$ so that the following hold. If $d_{GH}(B(x, r), B_{W}(o, r))<\epsilon' r$ for some metric cone $(W, o)$ and $0<r<\epsilon'$, then there exists a smooth function $u$ on $B(x, 2\delta_5r)$ with 
\begin{equation}0\leq u\leq C(n, v)\delta_5^2r^2; u_{\alpha\overline\beta}\geq c(n, v)g_{\alpha\overline\beta}>0;
\end{equation}
\begin{equation}
\min\limits_{y\in\partial B(x, \delta_5r)}u(y)> 4\sup\limits_{y\in B(x, \delta_6r)}u(y).
\end{equation}
\end{cor}
 \begin{proof}
 The proof is just a rescaled version of some arguments above.
 Let $0<\delta_4<<1$ depend only on $n, v$, to be determined.
Set $(Y', x', g') = (Y, x, \frac{g}{\delta_4^2r^2})$. Then $BK(Y')\geq -r^2\delta_4^2\geq -\epsilon'^2$, $vol(B(x', \frac{1}{\delta_4}))\geq \frac{c(n)v}{\delta_4^{2n}}$ and 
$$d_{GH}(B(x', \frac{1}{\delta_4}), B_{W}(o, \frac{1}{\delta_4}))<\frac{\epsilon'}{\delta_4}.$$
Observe Cheeger-Colding estimates (\ref{38}) and (\ref{39}) hold for the annulus $B(x', \frac{1}{5\delta_4})\backslash B(x', \frac{\delta_4}{100})$, if $\epsilon'$ is sufficiently small depending on $n, v, \delta_4$. By the same argument from (\ref{38}) to claim \ref{cl3}, we find a function $h_1$ on $B(x', \frac{1}{100\delta_4})$, $1>>\epsilon_0, \delta_4>0$ depending only on $n, v$ satisfying
$$(h_1)_{\alpha\overline\beta}\geq c(n, v)g_{\alpha\overline\beta}>0,$$
$$\min\limits_{y\in\partial B(x', \frac{1}{2000\delta_4})}h_1(y)> 4\sup\limits_{y\in B(x', \frac{\epsilon_0}{\delta_4})}h_1(y); 0\leq h_1(y)\leq C(n, v)\frac{1}{\delta_4^2}.$$ Now we freeze the value of $\delta_4=\delta_4(n, v)$ and $\epsilon'=\epsilon'(n, v, \delta_4) =\epsilon'(n, v)$.
We can think $h_1$ is defined on $B(x, \frac{r}{100})$. Take $u = \delta_4^2r^2h_1$, $\delta_5 = \frac{1}{2000}$, $\delta_6 = \epsilon_0$. This concludes the proof.
\end{proof}
\begin{cor}\label{cor-2}
Let $(M^n, x)$ be a complete K\"ahler manifold with $BK\geq -1$. Let $(X, (0, o)) = (\mathbb{C}^{n-1}, 0)\times (Z, o)$, where $(Z, o)$ is a real two dimensional cone with cone angle $\alpha$. Assume $vol(B(x, 1))\geq v>0$. Then there exist $0<\tilde{\epsilon}',  \delta_0<<1$ depending only on $n, v$ so that the following hold. If $0<r<\tilde{\epsilon}'$ and
$d_{GH}(B(x, r), B_X((0, o), r))<\tilde\epsilon'^2 r$, then there exists a holomorphic chart $(w_1, ..., w_n)$ on $B(x, \delta_0 r)$ such that
\begin{itemize}
\item $w_i(x) = 0(1\leq i\leq n)$.
\item  On $B(x, \delta_0 r)$, $|w_i|\leq C(n, v)r$. 
\item $(w_1,..., w_n)(B(x, \frac{1}{3}\delta_0r))\subset K_{\frac{5}{12}\delta_0r}$.
$(w_1,..., w_n)(B(x, \frac{4}{5}\delta_0r))\subset K_{\frac{8}{9}\delta_0r}$.  Recall that $K_r$ is defined in (\ref{-1000}).\item  $(w_1, ..., w_n)(B(x, \delta_0r))$ contains the domain $K_{\frac{9}{10}\delta_0r}$. $(w_1, ..., w_n)(B(x, \frac{2}{3}\delta_0r))$ contains the domain $K_{\frac{1}{2}\delta_0r}$. 
\end{itemize}
\end{cor}
\begin{proof}
Corollary \ref{cor-2} is a rescaled version of proposition \ref{prop2}.  First, note that if $\tilde\epsilon'$ is sufficiently small, then $\alpha>c(n, v)>0$. Set $(M', x', g') = (M, x, \frac{g}{\tilde\epsilon'^2 r^2})$. Then $BK(M')\geq -\tilde\epsilon'^2r^2\geq -\tilde\epsilon'^4$, $$d_{GH}(B(x', \frac{1}{\tilde\epsilon'}), B_{X}((0, o), \frac{1}{\tilde\epsilon'}))<\tilde\epsilon'.$$  According to proposition \ref{prop2}, if $\tilde\epsilon'=\tilde\epsilon'(n, v)$ is sufficiently small, then the following hold. There exist $\delta=\delta(n, v)>0$ and a holomorphic chart $(\tilde{w}_1, ..., \tilde{w}_n)$ on $B(x', \delta)$ with
\begin{itemize}
\item $\tilde{w}_i(x') = 0$ for $1\leq i\leq n$.
\item  On $B(x', \delta)$, $|\tilde{w}_i|\leq C(n, v)$. 
\item $(\tilde{w}_1,..., \tilde{w}_n)(B(x', \frac{1}{3}\delta))\subset K_{\frac{5}{12}\delta}$.
$(\tilde{w}_1,..., \tilde{w}_n)(B(x', \frac{4}{5}\delta))\subset K_{\frac{8}{9}\delta}$.  
\item  $(\tilde{w}_1, ..., \tilde{w}_n)(B(x', \delta))$ contains the domain $K_{\frac{9}{10}\delta}$. $(\tilde{w}_1, ..., \tilde{w}_n)(B(x', \frac{2}{3}\delta))$ contains the domain $K_{\frac{1}{2}\delta}$. 
\end{itemize}

We may think $\tilde{w}_i$ are functions on $B(x, \delta\tilde\epsilon' r)$. Set $w_j = \tilde\epsilon'r\tilde{w}_j$ for $1\leq j\leq n-1$, $w_n = (\tilde\epsilon'r)^{\frac{2\pi}{\alpha}}\tilde{w}_n$, $\delta_0 = \delta\tilde\epsilon'$. The proof is complete. \end{proof}

\begin{cor}\label{cor-1}
Under the assumptions of proposition \ref{prop2}, there exist $\rho=\rho(n, \gamma)>0$ and an open set $\Omega_x$ with $B(x, \delta a)\supset\Omega_x\supset B(x, \rho a)$, such that $\Omega_x$ biholomorphic to a ball in the Euclidean space. In particular, $\Omega_x$ is contractible.
\end{cor}

Now we prove corollary \ref{cor0}:
\begin{proof}
Assume $M$ is not simply connected. Let $\gamma$ be a smooth closed curve on $M$ which represents a nonzero element in $\pi_1(M)$. By the second variation of arc length, one finds that $\gamma$ cannot minimize the length in its free homotopy class. Thus we can take a sequence of smooth closed curves $\gamma_i\to \infty$ on $M$ with $[\gamma_i] = [\gamma]\in\pi_1(M)$ and the length $|\gamma_i|\leq |\gamma|$. Let $q_i\in \gamma_i$. Let $r_i = d(p, q_i)\to\infty$.

Consider the blow down sequence $(M_i, p_i) = (M, p, \frac{g}{(r_i)^2})$. By passing to a subsequence, we may assume $(M_i,  p_i)\to (X, p_\infty)$ in the pointed Gromov-Hausdorff sense. We may think that the $q_i\in\gamma_i\subset (M_i, p_i)$ and $q_i\to q_\infty\in (X, p_\infty)$. Then $d(p_\infty, q_\infty) = 1$ on $X$.
By Cheeger-Colding \cite{[CC2]}, $(X, p_\infty)$ is a metric cone. Thus the tangent cone at $q_\infty$ splits off a line. This means that given any $\epsilon>0$, there exists $\delta>0$ with $B(q_\infty, \delta)$ $\epsilon\delta$-Gromov-Hausdorff close to a ball in $(\mathbb{C}, 0)\times (Z, o)$ centered at $(0, o)$. Here $(Z, o)$ is a complex one dimensional cone with cone angle $\alpha>0$. Then for $i$ sufficiently large, $B(q_i, \delta)$ is also $\epsilon\delta$-Gromov-Hausdorff close to a ball in $(\mathbb{C}, 0)\times (Z, o)$ centered at $(0, o)$. We may assume $\epsilon$ is so small that corollary \ref{cor-1} can be applied. Then there exists a contractible neighborhood $\Omega_i$ of $q_i$ which contains a fixed size metric ball centered at $q_i$. 
As the length of $\gamma_i$ is converging to zero in $(M_i, p_i)$, eventually $\gamma_i\subset \Omega_i$. Hence $\gamma_i$ is contractible. Contradiction!
\end{proof}

\section{Separation of points}
\begin{prop}\label{prop3}
Let $v, R>0$. There exists $\epsilon'_1=\epsilon'_1(n, v)>0$ so that the following hold.
Let $(Y'^n, q')$ be a complete K\"ahler manifold with bisectional curvature lower bound $-\frac{(\epsilon'_1)^3}{R^2}$. Assume $vol(B(q',  \frac{R}{\epsilon'_1}))\geq \frac{vR^{2n}}{(\epsilon'_1)^{2n}}$.  Assume also \begin{equation}\label{72}d_{GH}(B(q', \frac{1}{\epsilon'_1}R), B_{W}(o, \frac{1}{\epsilon'_1}R))\leq \epsilon'_1 R\end{equation} for some metric cone $(W, o)$ centered at $o$. Then there exist $N' = N'(v, n)\in\mathbb{N}, 1 >\delta'_1> 5\delta'_2>c(v, n)>0$ and holomorphic functions $g'^1, ..., g'^{N'}$ on $B(q', \delta'_1R)$ with $g'^j(q') = 0$ and \begin{equation}\label{73}\min\limits_{x\in\partial B(q', \frac{1}{3}\delta'_1R)}\sum\limits_{j=1}^{N'}|g'^j(x)|^2> 2\sup\limits_{x\in B(q', \delta'_2R)}\sum\limits_{j=1}^{N'}|g'^j(x)|^2.\end{equation} Furthermore, for all $j$, \begin{equation}\label{74}\frac{\sup\limits_{x\in B(q', \frac{1}{2}\delta'_1R)}|g'^j(x)|^2}{\sup\limits_{x\in B(q', \frac{1}{3}\delta'_1R)}|g'^j(x)|^2}\leq C(n, v).\end{equation} \end{prop}
\begin{proof}
The proof is a minor modification of proposition $3$ in \cite{[L2]}.
The key is an induction argument which involves the stratification of singular sets.
Note we need to apply the three circle theorem \ref{thm-3}.
\end{proof}

The next corollary is a rescaled version of proposition \ref{prop3}.
\begin{cor}\label{cor-3}
Let $(Y^n, q)$ be a complete K\"ahler manifold with $BK\geq -1$ and $vol(B(q, 1))>v>0$. Then there exist $\epsilon_1=\epsilon_1(n, v)>0$ so that the following hold. Assume $$d_{GH}(B(q, r), B(o, r))<\epsilon^2_1r$$ for some metric cone $(W, o)$ and $0<r<\epsilon_1$. Then there exist $N= N(v, n)\in\mathbb{N}, 1 >\delta_1> 5\delta_2>c(v, n)>0$ and holomorphic functions $g^1, ..., g^N$ on $B(q, \delta_1r)$ with $g^j(q) = 0$ and $$\min\limits_{x\in\partial B(q, \frac{1}{3}\delta_1r)}\sum\limits_{j=1}^N|g^j(x)|^2> 2\sup\limits_{x\in B(q, \delta_2r)}\sum\limits_{j=1}^N|g^j(x)|^2.$$ Furthermore, for all $j$, $$\frac{\sup\limits_{x\in B(q, \frac{1}{2}\delta_1r)}|g^j(x)|^2}{\sup\limits_{x\in B(q, \frac{1}{3}\delta_1r)}|g^j(x)|^2}\leq C(n, v).$$\end{cor}
The proof is similar to corollary \ref{cor-2}. It suffices to scale the metric by $\frac{1}{\epsilon_1^2r^2}$.  We omit the details.
Now we come to the separation of points. The following proposition uses the same notations as in theorem \ref{thm1}.
\begin{prop}\label{prop4}
Let $x\in M_\infty$, $r(x) = d(x, p_\infty)$. There exist $\epsilon_2>0, \delta_3>0, 1>\gamma_1>0$ depending only on $n, r(x), v$ so that the following hold. Consider a sequence $x_i\to x$, $x_i\in M_i$. Let $(X, o)$ be a metric cone centered at $o$. If $0<R<\delta_3$ and $d_{GH}(B(x, R), B_X(o, R))<\epsilon_2 R$, then for sufficiently large $i$ and any two points $y^1_i\neq y^2_i\in B(x_i, \gamma_1R)$ with $d(y^1_i, y^2_i)>d>0$, there exists a holomorphic function $f_i$ on $B(x_i, 2\gamma_1R)$ with $f_i(y^1_i) = 0, f_i(y^2_i) = 1$ and $|f_i|\leq C(n, v, r(x), d, R)$. 
\end{prop}
\begin{remark}
The point is that all constants are independent of $i$. Thus, limit functions separate near points on $M_\infty$.
\end{remark}
\begin{proof}
The volume comparison theorem says $vol(B(x_i, 1))\geq c(n, v, r(x))>0$.
By corollary \ref{cor2}, we can find small positive constants $\gamma_0, \gamma_1, \epsilon_2, \delta_3$ depending only on $n, v, r(x)$ and a function $h_i$ with \begin{equation}\label{-3}C(n, v, r(x))\gamma_0^2R^2 \geq h_i\geq 0\end{equation} on $B(x_i, \gamma_0R)$. Moreover, 
\begin{equation}\label{75}(h_i)_{\alpha\overline\beta}\geq c(n, v, r(x))g_{\alpha\overline\beta}>0,\end{equation}
\begin{equation}\label{76}\min\limits_{y\in\partial B(x_i, \frac{\gamma_0R}{2})}h_i(y)> 4\sup\limits_{y\in B(x_i, 3\gamma_1R)}h_i(y).\end{equation}
Let $\Omega_i$ be the connected component of $\{z\in B(x_i, \frac{\gamma_0R}{2})|h_i(z)<2\sup\limits_{y\in B(x_i, 3\gamma_1R)}h_i(y)\}$ containing $B(x_i, 3\gamma_1R)$. Then $\Omega_i$ is relatively compact in $B(x_i, \frac{1}{2}\gamma_0R)$. $\Omega_i$ is a Stein manifold.

Let $\epsilon''>0$ be a small constant depending only on $n, v, r(x)$.
For any point $y$ in $B(x_i, R)$, there exists $\frac{d}{10}>r_y>0$ with \begin{equation}\label{77}d_{GH}(B(y, r_y), B_{X_y}(o_y, r_y))<\epsilon'' r_y.\end{equation} Here $(X_y, o_y)$ is a metric cone. We may assume $\epsilon''$ and $r_y$ are so small that corollary \ref{cor-3} can be applied. Now we freeze the value of $\epsilon''$.
By Gromov compactness theorem, we may also assume 
\begin{equation}\label{78}\frac{d}{10}>r_y>c(n, v, d, r(x))>0.\end{equation}
Thus for $j=1, 2$, there exist $N = N(v, n, r(x))\in\mathbb{N}, 1 >\delta_1> 5\delta_2>c(v, n, r(x))>0$ and holomorphic functions 
$g_{ij}^1, ..., g_{ij}^N$ on $B(y_i^j, \delta_1r_{y^j_i})$ with $g_{ij}^s(y_i^j) = 0$ and \begin{equation}\label{79}\min\limits_{z\in\partial B(y_i^j, \frac{1}{3}\delta_1r_{y_i^j})}\sum\limits_{s=1}^N|g_{ij}^s(z)|^2> 2\sup\limits_{z\in B(y_i^j, \delta_2r_{y_i^j})}\sum\limits_{s=1}^N|g_{ij}^s(z)|^2.\end{equation} Furthermore, for all $s$, \begin{equation}\label{80}\frac{\sup\limits_{z\in B(y_i^j, \frac{1}{2}\delta_1r_{y_i^j})}|g_{ij}^s(z)|^2}{\sup\limits_{z\in B(y^j_i, \frac{1}{3}\delta_1r_{y_i^j})}|g^s_{ij}(z)|^2}\leq C(n, v, r(x)).\end{equation}

 By normalization, we can also assume \begin{equation}\label{81}\max_s\sup\limits_{z\in B(y_i^j, \delta_2r_{y_i^j})}|g^s_{ij}(z)| = 2.\end{equation} 

Note by three circle theorem \ref{thm-3} and (\ref{80}), \begin{equation}\label{82}\max\limits_{z\in B(y_i^j, \frac{1}{2}\delta_1r_{y_i^j})}|g^s_{ij}(z)| \leq C(n, v, r(x)).\end{equation} Set \begin{equation}\label{83}F^j_i = \sum\limits_{s=1}^N|g^s_{ij}|^2.\end{equation} Let $\lambda$ be a standard cut-off function: $\mathbb{R}^+\to\mathbb{R}^+$ given by $\lambda(t) = 1$ for $0\leq t\leq 1$; $\lambda(t) = 0$ for $t\geq 2$; $|\lambda'|, |\lambda''|\leq C(n)$.
Consider \begin{equation}\label{84}v^j_i(z) = 4n\log F^j_i(z)\lambda(F^j_i(z))\end{equation} on $B(y^j_i, \frac{1}{3}\delta_1r_{y^j_i})$. By (\ref{79}) and (\ref{81}), $v^j_i$ is compactly supported on $B(y^j_i, \frac{1}{3}\delta_1r_{y^j_i})$. We extend it to zero outside.
\begin{lemma}\label{lm4} \begin{equation}\label{85}\sqrt{-1}\partial\overline\partial v^j_i(z)\geq -C(n, v, R, r(x), d)\omega_i\end{equation} where $\omega_i$ is the K\"ahler metric on $M_i$. Moreover, $e^{-v^j_i}$ is not locally integrable at $y^j_i$.\end{lemma}
\begin{proof}
The proof is similar to lemma $1$ in \cite{[L2]}. We skip it here. Note (\ref{78}) is crucial.
\end{proof}

 Therefore, there exists $\xi = \xi(n, v, R, d, r(x))>0$ with \begin{equation}\label{86}\sqrt{-1}\partial\overline\partial\psi_i\geq 5(n+1)\omega_i\end{equation} on $\Omega_i$, where 
 $\psi_i = \xi h_i+v^1_i+v^2_i$. Then $\psi_i\leq C(n, v, r(x), d, R)$.
 
Now consider a function $\mu_i(z) = 1$ for $z\in B(y_i^1, \frac{d}{4})$; $\mu_i$ has compact support in $B(y_i^1, \frac{d}{2})$; $|\nabla\mu_i|\leq C(n, d)$. By theorem \ref{thm-2}, we can solve the equation $\overline\partial w_i = \overline\partial \mu_i$ in $\Omega_i$ (defined below (\ref{76})) with \begin{equation}\label{87}\int_{\Omega_i}|w_i|^2e^{-\psi_i}\leq \int_{\Omega_i}|\overline\partial \mu_i|^2e^{-\psi_i}\leq C(n, v, R, r(x), d).\end{equation} Set $f_i = \mu_i - w_i$ on $B(x, 3\gamma_1R)$. By lemma \ref{lm4}, $w_i(y_i^1) = w_i(y_i^2) = 0$.  As $\mu_i(y^1_i) = 1, \mu_i(y^2_i) = 0$, $f_i(y^1_i) = 1, f_i(y^2_i) = 0$.
Note $\int_{\Omega_i}|f_i|^2\leq 2\int_{\Omega_i}(|\mu_i|^2+|w_i|^2)\leq C(n, v, R, d, r(x))$.
 By mean value inequality, $|f_i|\leq C(n, v, R, d, r(x))$ on $B(x_i, 2\gamma_1R)$. 
\end{proof}

\section{Construction of local coordinates on the limit space}
Recall $\mathcal{W}\mathcal{E}_{2n-2} = \{x\in M_\infty|$there exists a tangent cone splitting off $\mathbb{R}^{2n-2}\}$.
For $x\in \mathcal{W}\mathcal{E}_{2n-2}$,  let $C_x$ be a tangent cone at $x$ which splits off $\mathbb{R}^{2n-2}$.  Then \begin{equation}\label{88}C_x(0, o) = (\mathbb{R}^{2n-2}, 0)\times (Z_x, o)\end{equation} where $Z_x$ is a real two dimensional cone with cone angle $\alpha$ satisfying $2\pi\geq \alpha\geq \gamma$. Here $\gamma=\gamma(r(x), v, n)>0$, $r(x)=d(x, p_\infty)$. For sufficiently large $i$, we can find $\epsilon_2>0, 1>>r'_x>0, x_i\in M_i, x_i\to x$ and \begin{equation}\label{89}d_{GH}(B(x_i, r'_x), B_{C_x}((0, o), r'_x))<\epsilon_2 r'_x\end{equation} so that the conditions of proposition \ref{prop4} are satisfied. Let $\gamma_1=\gamma_1(n, v, r(x))>0$ be the constant in proposition \ref{prop4}.
It is straightforward to see that 
\begin{equation}\label{-10}
d_{GH}(B(x_i, \gamma_1r'_x), B_{C_x}((0, o), \gamma_1r'_x))<10\epsilon_2 r'_x
\end{equation}

By shrinking the values of $r'_x$ and $\epsilon_2$ if necessary, we may assume that corollary \ref{cor-2} can be applied to $B(x_i, \gamma_1r'_x)$.
Then there exists a holomorphic chart $(w^i_{x1}, ..., w^i_{xn})$ on $B(x_i, \delta_0\gamma_1r'_x)$ for $\delta_0 = \delta_0(n, v, r(x))<<1$. Also \begin{equation}\label{90}|w^i_{xs}|\leq C(n, v, r(x))r'_x.\end{equation} Gradient estimate says \begin{equation}\label{-11}
|dw^i_{xs}|\leq C(n, v, r(x))\end{equation} on $B(x_i, \frac{5}{6}\delta_0\gamma_1r'_x)$. 
Now Arzela-Ascoli theorem implies that a subsequence of $w^i_{xs}$ converges uniformly to $w^\infty_{xs}$ on $B(x, \frac{5\delta_0\gamma_1r'_x}{6})$. 

\begin{lemma}\label{lm5}
$(w^\infty_{x1}, ..., w^\infty_{xn})$ is injective on $B(x, \frac{4}{5}\delta_0\gamma_1r'_x)$. 
\end{lemma}
\begin{proof}
Assume $q_1\neq q_2\in B(x, \frac{4}{5}\delta_0\gamma_1r'_x)$ and $w^\infty_{xs}(q_1) = w^\infty_{xs}(q_2)$ for $1\leq s\leq n$. Let $d= d(q_1, q_2)>0$. Consider sequences $M_i\ni q^i_1\to q_1, M_i\ni q^i_2\to q_2$. We may assume $d(q^i_1, q^i_2)>\frac{d}{2}>0$. According to proposition \ref{prop4}, we find $f_i$ holomorphic on $B(x_i, 2\gamma_1r'_x)$ with \begin{equation}\label{91}f_i(q^i_1) = 0; f_i(q^i_2) = 1; |f_i|\leq C(n, r'_x, v, r(x), d).\end{equation} As $w^i_{xs}$ is a holomorphic chart on $B(x_i,\delta_0\gamma_1r'_x)$, we may write $f_i(z) = g_i(w^i_{x1}(z), ..., w^i_{xn}(z))$ on $B(x_i, \delta_0\gamma_1r'_x)$.

By corollary \ref{cor-2}, the image of $(w^i_{x1}, ..., w^i_{xn})$ contains $K_{\frac{9}{10}\delta_0\gamma_1r'_x}$ in $\mathbb{C}^n$. Then $g_i$ is well defined on $K_{\frac{9}{10}\gamma_1\delta_0r'_x}$. From the standard Cauchy integral estimate, we have
\begin{claim}\label{cl7}
$|\frac{\partial g_i}{\partial w^i_{xs}}|\leq C(n, v, r'_x, r(x), d)$ on $K_{\frac{8\delta_0\gamma_1r'_x}{9}}$. In particular, $g_i$ has a convergent subsequence. Also note by corollary \ref{cor-2},
$(w^i_{x1}, ..., w^i_{xn})(B(x_i, \frac{4}{5}\delta_0\gamma_1r'_x))\subset K_{\frac{8}{9}\gamma_1\delta_0r'_x}$.
\end{claim}

On the one hand, $f_i$ has a convergent subsequence, say $f_i\to f_\infty$ uniformly on $B(x, \frac{3\gamma_1r'_x}{2})$. Therefore, $f_\infty(q_1) = 0, f_\infty(q_2) = 1$.
On the other hand, by claim \ref{cl7} and that $w^i_{xs}$ are convergent, after taking further subsequence, $f_i=g_i(w^i_{x1}, ... , w^i_{xn})$ converges uniformly to $f_\infty=g_\infty(w^\infty_{x1}, ..., w^\infty_{xn})$ on $B(x, \frac{4}{5}\delta_0\gamma_1r'_x)$. Then $$f_\infty(q_1) =g_\infty(w^\infty_{x1}(q_1), ..., w^\infty_{xn}(q_1)) = g_\infty(w^\infty_{x1}(q_2), ..., w^\infty_{xn}(q_2)) = f_\infty(q_2).$$
This is a contradiction.

\end{proof}
Let 
$\Omega_\infty = (w^\infty_{x1}, ..., w^\infty_{xn})^{-1}(K_{\frac{1}{2}\delta_0\gamma_1r'_x}).$
\begin{claim}\label{dl-100}
$(w^\infty_{x1}, ..., w^\infty_{xn})$ is a homeomorphism from $\Omega_\infty$ to $K_{\frac{1}{2}\delta_0\gamma_1r_x}$. 
\end{claim}
\begin{proof}
$\Omega_\infty$ is open in $M_\infty$, as $(w^\infty_{x1}, ..., w^\infty_{xn})$ is continuous.
According to corollary \ref{cor-2}, $(w^i_{x1}, ..., w^i_{xn})^{-1}(K_{\frac{1}{2}\delta_0\gamma_1r'_x})\subset B(x_i, \frac{2}{3}\delta_0\gamma_1r'_x)$. 
Then $\Omega_\infty\subset B(x, \frac{3}{4}\delta_0\gamma_1r'_x)$. Lemma \ref{lm5} implies that
$(w^\infty_{x1}, ..., w^\infty_{xn})$ is injective on $\Omega_\infty$. It suffices to prove the surjectivity. For any $y\in K_{\frac{1}{2}\delta_0\gamma_1r_x}$,
let $z_i = (w^i_{x1}, ..., w^i_{xn})^{-1}(y)\in B(x_i, \frac{2}{3}\delta_0\gamma_1r'_x)$. We may assume a subsequence of $z_i$ converges to $z\in B(x, \frac{3}{4}\gamma_1\delta_0r'_x)$. Then $y= (w^\infty_{x1}(z), ..., w^\infty_{xn}(z))$. This concludes the proof.
\end{proof}
Corollary \ref{cor-2} says $(w^i_{x1}, ..., w^i_{xn})(B(x_i, \frac{1}{3}\gamma_1\delta_0r'_x))\subset K_\frac{5\delta_0\gamma_1r_x'}{12}$. Therefore $$B(x, \frac{1}{3}\gamma_1\delta_0r'_x)\subset\Omega_\infty.$$ We conclude that $(w^\infty_{x1}, ..., w^\infty_{xn})$ is a coordinate system on $B(x, \frac{1}{3}\gamma_1\delta_0r'_x)$.
Let \begin{equation}\label{-12}\tilde\gamma_1 = \gamma_1\delta_0.\end{equation} Note $\tilde\gamma_1$ depends only on $n, v, r(x)$. Set
\begin{equation}\label{92}G = \bigcup\limits_{x\in\mathcal{W}\mathcal{E}_{2n-2}}B(x, \frac{1}{5}\tilde\gamma_1r'_x).\end{equation} Then $G$ is open. The complement has codimension at least $4$ by theorem \ref{thm-1} and theorem \ref{thm-4}.
Take a locally finite covering of $G$, say \begin{equation}\label{93}G =\bigcup_{j\in\mathbb{N}}B(x^j, \frac{\tilde\gamma_1r'_{x^j}}{5}).\end{equation} By taking a subsequence, we may assume that $w^i_{x^j_is}$ converge to $w^\infty_{x^js}$ for $j\in\mathbb{N}$. Here $M_i\ni x^j_i\to x^j$.
\begin{claim}\label{cl8}
$(w^\infty_{x^j1}, ..., w^\infty_{x^jn})$ form a holomorphic atlas on $G$.
\end{claim}
\begin{proof}
It suffices to prove the transition functions are holomorphic. One can just look at the transition functions on $M_i$ for charts given by $w^i_{x^j_is}$. By Cauchy estimates as in claim \ref{cl7}, one proves that the transition functions are holomorphic with uniform bound. Thus their limits are still holomorphic.
\end{proof}
From claim \ref{cl8}, $G$ has a holomorphic structure. Let $x\in M_\infty$, $M_i\ni x_i$ and $x_i\to x$. Let $r_x, \epsilon_2>0$ satisfy $d_{GH}(B(x_i, r_x), B_X(o, r_x))<\epsilon_2 r_x$ for some metric cone $(X, o)$. We assume proposition \ref{prop4} is satisfied. We have the following proposition.
\begin{prop}\label{prop-10}
For any $y\in B(x, \frac{1}{2}\gamma_1r_x)\cap G$, there exist $n$ sequences of holomorphic functions $\lambda^i_j(1\leq j\leq n)$ on $B(x_i, \gamma_1r_x)$ so that $\lambda^i_j\to \lambda^\infty_j$ uniformly on $B(x, \frac{1}{2}\gamma_1r_x)$ and $(\lambda^\infty_1, ..., \lambda^\infty_n)$ forms a holomorphic coordinate around $y$.
\end{prop}
\begin{proof}
By the definition of $G$, (\ref{90}), (\ref{-11}) and lemma \ref{lm5}, 
we can find a sequence $y_i\in M_i$ with $y_i\to y$ so that the following hold.
\begin{itemize}
\item there exist holomorphic charts $(w^i_1, ..., w^i_n)$ on $B(y_i, 5\delta)$ for some $\delta>0$;
\item $w^i_j\to w^\infty_{j}$ uniformly on 
$B(y, 4\delta)$;
\item $(w^\infty_1, .., w_n^\infty)$ is a holomorphic chart on $B(y, 4\delta)\subset G$;
 \item $|w^i_j|$ is uniformly bounded on $B(y_i, 5\delta)$ for all $i$. Say $|w^i_j|\leq C$;
\item $B(y_i, 10\delta)\subset B(x_i, \frac{1}{2}\gamma_1r_x)$, $B(y, 10\delta)\subset B(x, \frac{1}{2}\gamma_1r_x)$.
\item $w^i_j(y_i) = 0$ for all $i$ and $j$.
\item For sufficiently large $i$, $(w^i_1, ..., w^i_n)(B(y_i, \delta))\supset B_{\mathbb{C}^n}(0, \delta')$ for some $\delta'>0$.
\item  $(w^\infty_1, ..., w^\infty_n)(B(y, \delta))\supset B_{\mathbb{C}^n}(0, \delta')$.
\item There exists $\delta''>0$ with  $(w^i_1, ..., w^i_n)(B(y, \delta''))\subset B_{\mathbb{C}^n}(0, \frac{\delta'}{2})$.
\end{itemize}

Consider smooth cut-off functions $\tau_i$ with $\tau_i = 1$ in $B(y_i, 2\delta)$, $\tau_i$ have compact support in $B(y_i, 3\delta)$, $|\nabla\tau_i|\leq \frac{20}{\delta}$.
Let \begin{equation}\label{-50}h^i_j = \tau_iw^i_j.\end{equation} Note $h^i_j$ is holomorphic on $B(y_i, 2\delta)$. Recall the function $h_i$ in proposition \ref{prop4}(replace $R$ by $r_x$) satisfies (\ref{-3}), (\ref{75}) and (\ref{76}). Also recall the Stein manifold $\Omega_i$ right below (\ref{76}). Let $\lambda$ be a standard cut-off function: $\mathbb{R}^+\to\mathbb{R}^+$ given by $\lambda(t) = 1$ for $0\leq t\leq 1$; $\lambda(t) = 0$ for $t\geq 2$; $|\lambda'|, |\lambda''|\leq C(n)$. Define \begin{equation}\label{-51}F_i = \sum\limits_{j=1}^n|w^i_j|^2.\end{equation} Given a constant $\xi>0$, set
\begin{equation}\label{-52}\Psi_i= \xi h_i+8n\log(F_i)\lambda(\frac{4F_i}{\delta'^2}).\end{equation} Extend $\log(F_i)\lambda(\frac{4F_i}{\delta'^2})$ to zero outside $B(y_i, \delta)$.
Similar as in (\ref{86}), we can find a large constant $\xi$ independent of $i$ with 
\begin{equation}\label{-53}(\Psi_i)_{\alpha\overline\beta}\geq 5(n+1)(g_i)_{\alpha\overline\beta}\end{equation} on $\Omega_i$. Here $g_i$ is the K\"ahler metric on $M_i$.
We solve the $\overline\partial$-problem on $\Omega_i$ \begin{equation}\label{-54}\overline\partial f^i_j =\overline\partial h^i_j\end{equation} satisfying \begin{equation}\label{-55}\int_{\Omega_i}|f^i_j|^2e^{-\Psi_i}\leq \int_{\Omega_i}|\overline\partial h^i_j|^2e^{-\Psi_i} \end{equation}
Below $C_1, C_2, ...$ will be large constants independent of $i$.
It is straightforward to verify that \begin{equation}\label{-56}\int_{\Omega_i}|\overline\partial h^i_j|^2e^{-\Psi_i}\leq C_1.\end{equation}
On $B(y_i, 3\delta)$, we can write $f^i_j = f^i_j(w^i_1, ..., w^i_n)$.  We have \begin{equation}\label{-57}\int_{\Omega_i}|f^i_j|^2e^{-\Psi_i}\leq C_1,\end{equation}
Note for each fixed $i$, the volume form $(\frac{1}{2\sqrt{-1}})^ndw^i_1\wedge\overline{dw^i_1}\wedge\cdot\cdot\cdot\wedge dw^i_n\wedge\overline{dw^i_n}$ is equivalent to the volume form of $g_i$ on $B(y_i, 3\delta)$.
Since $w^i_j(y_i) = 0$, the local integrability near $y_i$ implies \begin{equation}\label{-61}\frac{\partial f^i_j}{\partial w^i_s}(0, ..., 0) = 0\end{equation} for $1\leq j, s\leq n$. Here $(0, .., 0) = (w^i_1(y_i), ..., w^i_n(y_i))$.
Set $\lambda^i_j = h^i_j-f^i_j$ on $\Omega_i\supset B(x, 2\gamma_1r_x)$. Note $h^i_j$ is uniformly bounded. (\ref{-57}) implies that 
\begin{equation}\label{-58}
\int_{\Omega_i}|\lambda^i_j|^2\leq C_2.
\end{equation}
 Mean value inequality implies that $|\lambda^i_j|\leq C_3$ on $B(x_i, \gamma_1r_x)$.
As $\lambda^i_j$ is holomorphic, by taking subsequences, we may assume $\lambda^i_j\to \lambda^\infty_j$ uniformly on $B(x, \frac{1}{2}\gamma_1r_x)$. Recall $h^i_j = w^i_j$ on $B(y_i, 2\delta)$. According to (\ref{-61}), \begin{equation}\label{-59}\frac{\partial\lambda^i_j}{\partial w^i_s}(0, ..., 0) = \delta_{js}. \end{equation}
Letting $i\to\infty$, we obtain that \begin{equation}\label{-60}\frac{\partial\lambda^\infty_j}{\partial w^\infty_s}(0, ..., 0) = \delta_{js}. \end{equation}
This proves that $(\lambda^\infty_1, ..., \lambda^\infty_n)$ forms a holomorphic coordinate around $y$.\end{proof}

\section{Holomorphic functions on limit space}

\begin{definition}
Let $\mathcal{F}$ be the sheaf on $M_\infty$ so that for any open set $U$ of $M_\infty$, $\Gamma(U, \mathcal{F})$ consists of all holomorphic functions on $U\cap G$ which are locally bounded on $U$.  \end{definition}

Below we use the same notions as in proposition \ref{prop4}. We shall replace $R$ by $r_x$. Then, for some metric cone $(X, o)$,
\begin{equation}\label{94}d_{GH}(B(x_i, r_x), B_X(o, r_x))<\epsilon_2 r_x.\end{equation}
\begin{lemma}\label{lm6}
Let $x\in M_\infty$. Consider $x_i\in M_i$ with $x_i\to x$. If $f\in\Gamma(B(x, r_x), \mathcal{F})$, then there exists $f_i$ holomorphic and uniformly bounded on $B(x_i, \frac{\gamma_1}{2}r_x)$ so that $f_i\to f$ uniformly on $B(x, \frac{\gamma_1}{2}r_x)\cap G$. Conversely, if $f_i$ is holomorphic on $B(x_i, r_x)$ and $f_i\to f$ uniformly, then $f_{B(x, \frac{\gamma_1}{2}r_x)}\in\Gamma(B(x, \frac{\gamma_1}{2}r_x), \mathcal{F})$.
\end{lemma}

\begin{proof}
Let $f_i$ be holomorphic on $B(x_i, \gamma_1r_x)$ and $f_i\to f$ uniformly on $B(x, \gamma_1r_x)$. One just need to prove $f$ is holomorphic on $G\cap B(x, \frac{1}{2}\gamma_1r_x)$. This follows from the same argument as in lemma \ref{lm5}.

Now assume $f\in\Gamma(B(x, r_x), \mathcal{F})$. We shall use some cut-off argument similar as in \cite{[DS]}. Let $\Sigma = M_\infty\backslash G$ and $\Sigma_i$ be the preimage of $\Sigma$ in $M_i$ by the Gromov-Hausdorff approximation (here is $\Sigma_i$ need not be precisely defined). We are going to transplant $f$ to $B(x_i, \frac{3}{4}r_x)\backslash B(\Sigma_i, d_i))$ for some $d_i\to 0$. By modifying the locally finite covering $B_j = B(x^j, \frac{1}{5}\tilde\gamma_1r'_{x^j})$ in (\ref{93}), we can find a partition of unity of $G$, $\varphi_j$, subordinate to $B_j$,  smooth with respect to holomorphic structure on $G$. On $B_j$, we may write \begin{equation}\label{95}\varphi_j = \varphi_j(w^\infty_{x^j1}, ..., w^\infty_{x^jn}, \overline{w^\infty_{x^j1}}, ..., \overline{w^\infty_{x^jn}}).\end{equation} Define \begin{equation}\label{96}\varphi_{ij} = \varphi_j(w^i_{x_i^j1}, ..., w^i_{x_i^jn}, \overline{w^i_{x_i^j1}}, ..., \overline{w^i_{x_i^jn}}).\end{equation} Here we use the notations right below (\ref{93}). Then on any compact set $K$ of $G$, $\varphi_{ij}\to\varphi_{j}$ uniformly. If we replace $\varphi_{ij}$ by $\frac{\varphi_{ij}}{\sum\limits_{s}\varphi_{is}}$, then $\sum\limits_{j}\varphi_{ij} = 1$ on $K_i$ for sufficiently large $i$. Here $K_i$ is the preimage of $K$ in $M_i$.

$G$ is dense in $B'_j=B(x^j,\frac{1}{3}\tilde\gamma_1r'_{x^j})$. Note by the sentence above (\ref{-12}), there is a holomorphic chart on $B'_j$.  Then $f$ extends to a holomorphic function on $B'_j\cap B(x, r_x)$.  It is clear the extension glues on the intersections of $B'_j$.

On $B'_j\cap B(x, r_x)$, write $f = f^j(w^\infty_{x^j1}, ..., w^\infty_{x^jn})$ where $f^j$ is holomorphic. Define 
\begin{equation}\label{97}f_{ij} = f^j(w^i_{x_i^j1}, ..., w^i_{x_i^jn})\end{equation} 
on $B(x^j_i, \frac{1}{5}\tilde\gamma_1r'_{x_i^j})\cap B(x_i, \frac{3}{4}r_x)$. Note this is well defined for sufficiently large $i$. Also $f_{ij}\to f$ on $B_j\cap B(x, \frac{3}{4}r_x)$.
Now define a function \begin{equation}u_i = \sum\limits_{j}\varphi_{ij}f_{ij}.\end{equation} It is clear $u_i\to f$ uniformly on each compact set of $G\cap B(x, \frac{3}{4}r_x)$.
\begin{claim}\label{cl9}
$|\overline\partial u_i|\to 0$ uniformly on each compact set $K$ of $G\cap B(x, \frac{3}{4}r_x)$.  $|du_i|$ is uniformly bounded on the preimage of $K$ in $M_i$.
\end{claim}
\begin{proof}
By definition, $f_{ij}$ are holomorphic. Let $z\in B_{j_0}$. Consider a sequence $z_i\in B(x^{j_0}_i, \frac{1}{5}\tilde\gamma_1r'_{x_i^{j_0}})$ with $z_i\to z$. Thus \begin{equation}\label{99}\begin{aligned}|\overline\partial(\sum\limits_j\varphi_{ij}(z_i)f_{ij}(z_i))| &= |\sum\limits_{j}f_{ij}(z_i)\overline\partial\varphi_{ij}(z_i)|\\&\leq \sum\limits_{j}|f_{ij}(z_i)-f(z)||\overline\partial\varphi_{ij}(z_i)|+|f(z)||\overline\partial(\sum\limits_{j}\varphi_{ij})(z_i)|\\&\to 0.\end{aligned}\end{equation} The second assertion follows similarly.
\end{proof}

By the same argument as in proposition $3.5$ of \cite{[DS]}, we can find a smooth cut off function $\beta_i$ on $B(x_i, r_x)$, satisfying $1-\beta_i$ has compact support in 
a $\Phi(\frac{1}{i})$-neighborhood of $\Sigma_i$; equals $1$ in a small neighborhood of $\Sigma_i$; $0\leq\beta_i\leq 1$; $\int_{B(x_i, r_x)}|\nabla\beta_i|^2\to 0$. 
We may also assume that $\beta_i\to 1$ sufficiently slow outside $\Sigma$.
Define the function $g_i = u_i\beta_i$. Then we can make that on $B(x_i, \frac{2}{3}r_x)$,
\begin{equation}\label{-15}
|g_i|\leq 2\sup\limits_{B(x, \frac{3}{4}r_x)\cap G}|f|+1
\end{equation}
Routine calculation shows
\begin{claim}\label{cl10}
$\int_{B(x_i, \frac{2}{3}r_x)}|\overline\partial g_i|^2\to 0$. 
\end{claim}

Let the function $h_i$ satisfy (\ref{-3}), (\ref{75}) and (\ref{76}) with $R$ replaced by $r_x$. Let $C=C(n, v, r(x))>0$ satisfy \begin{equation}(Ch_i)_{\alpha\overline\beta}\geq 4(n+1)g_{\alpha\overline\beta}>0.\end{equation}
Let $\Omega_i$ be the connected component of $\{z\in B(x_i, \frac{\gamma_0r_x}{2})|h_i(z)<2\max\limits_{y\in B(x_i, 3\gamma_1r_x)}h_i(y)\}$ containing $B(x_i, 3\gamma_1r_x)$. Then $\Omega_i$ is relatively compact in $B(x_i, \frac{1}{2}\gamma_0r_x)\subset B(x_i, \frac{2}{3}r_x)$ and $\Omega_i$ is a Stein manifold.
Now we solve the $\overline\partial$-problem \begin{equation}\label{100}\overline\partial g'_i =\overline\partial g_i\end{equation} on $\Omega_i\supset B(x_i, 3\gamma_1r_x)$ with 
\begin{equation}\label{101}\int_{\Omega_i}|g'_i|^2e^{-Ch_i}\leq\int_{\Omega_i}|\overline\partial g_i|^2e^{-Ch_i}\to 0.\end{equation} Therefore \begin{equation}\label{-4}\int_{\Omega_i}|g'_i|^2\to 0.\end{equation} Then by (\ref{-15}), the holomorphic function $f_i = g_i-g'_i$ satisfies \begin{equation}\label{102}\int_{B(x, 3\gamma_1r_x)}|f_i|^2\leq C(n, v, r(x))\sup\limits_{B(x, \frac{3}{4}r_x)}(1+|f|^2).\end{equation} Mean value inequality and the gradient estimate imply \begin{equation}\label{103}|df_i|, |f_i|\leq C(n, v, r(x), r_x)\sup\limits_{B(x, \frac{3}{4}r_x)}(1+|f|)\end{equation} on $B(x_i, \gamma_1r_x)$. For any sequence $M_i\ni z_i\to z\in G\cap B(x, \frac{3}{4}\gamma_1r_x)$, $dg_i(z_i) = du_i(z_i)$ for all sufficiently large $i$. Then by claim \ref{cl9}, $|dg'_i(z_i)|\leq |du_i(z_i)|+|df_i(z_i)|$ which has an upper bound independent of $i$. By (\ref{-4}), we obtain that $|g'_i|\to 0$ on each compact set of $G\cap B(x, \gamma_1r_x)$. That is, $f_i\to f$ uniformly on each compact set of $G\cap B(x, \gamma_1r_x)$. The convergence must be uniform on $G\cap B(x, \frac{1}{2}\gamma_1r_x)$, since $f_i$ is bounded and equicontinuous. 
\end{proof}

\begin{cor}\label{cor3}
Let $U$ be an open set of $M_\infty$ and $f\in \Gamma(U, \mathcal{F})$. Then $f$ extends to a continuous function on $U$.
\end{cor}
\begin{proof}
The problem is local. For $x\in U$, we can find $r_x$ satisfying the conditions of proposition \ref{prop4}($r_x$ replaces $R$) and $B(x, 2r_x)\subset U$. The corollary follows from the first statement of lemma \ref{lm6}. 
\end{proof}

\section{Completion of the proof of theorem \ref{thm1}}
In this section, we shall apply some localized argument in \cite{[DS]}.
Given $x\in M_\infty$, consider a sequence $x_i\in M_i$ converging to $x$. 
We still follow the notations in proposition \ref{prop4} with $R$ replaced by $r_x$. Then, $r_x$ satisfies (\ref{94}). We may also assume $r_x\geq c(n, v, r(x))>0$ by Gromov compactness theorem.

By applying proposition \ref{prop4} and Gromov compactness theorem repeatedly,  we can find some $m=N_0(n, v, r(x))$, $M=M(n, v, r(x))>0$, holomorphic functions $g^s_i$ on $B(x_i, \gamma_1r_x)(1\leq s\leq N_0)$ with
\begin{equation}\label{-16}
g^s_i(x_i) =0; |g^s_i|\leq M(n, v, r(x));\end{equation}
\begin{equation}\label{-17}
\min\limits_{y\in \partial B(x_i, \frac{1}{3}\gamma_1r_x)}(\sum\limits_{s=1}^{m}|g^s_i(y)|^2)^\frac{1}{2}\geq 2.
\end{equation}
This merely means that we separate $\partial B(x_i, \frac{1}{3}\gamma_1r_x)$ from $x_i$.
 Define $F^{m}_i = (g^1_i, ..., g^{m}_i)$. Below we will add more functions. That is, we increase the value $m$. By passing to subsequences, we always assume that $g^s_i\to g^s, F^{m}_i\to F^{m}$ on $B(x, \frac{1}{2}\gamma_1r_x)$.  We also assume that (\ref{-16}) is true for all $m\geq N_0$.
 
Let $|\cdot|$ be the standard norm on $\mathbb{C}^m$. 
By gradient estimate, $|dg^s_i|\leq C(n, v, r(x))$ on $B(x_i, \frac{1}{2}\gamma_1r_x)$. Then there exists $\gamma_2=\gamma_2(n, v, r(x))$ so that 
\begin{equation}\label{-18}
|F^{N_0}_i(y)|< \frac{1}{10}
\end{equation}
for $y\in B(x_i, \gamma_2r_x)$. By applying proposition \ref{prop4} and Gromov compactness theorem again, we find $\tau=\tau(n, v, r(x))>0$ and $N_1=N_1(n, v, r(x))>N_0$ with 
\begin{equation}\label{-19}
|F^{N_1}_i(y)|<\frac{1}{5}, y\in B(x_i, \gamma_2r_x),
\end{equation}
\begin{equation}\label{-20}
 |F^{N_1}_i|\geq 2\tau
\end{equation}
on $B(x_i, \gamma_1r_x)\backslash B(x_i, \frac{1}{2}\gamma_2r_x)$. This can be achieved by rescaling $g^s_i(s>N_0)$ by small factors.
We summarize the constructions above. For $m=N_1$, conditions (a)-(c) are valid: 

(a). $g^s_i(x_i) =0$ and $|g^s_i|\leq M(n, v, r(x))$ on $B(x_i, \gamma_1r_x)$ for $1\leq s\leq m$;

(b). $\min\limits_{y\in \partial B(x_i, \frac{1}{3}\gamma_1r_x)} |F^{m}_i(y)|\geq 2$; $|F^{m}_i(y)|<\frac{1}{5}$ for $y\in B(x_i, \gamma_2r_x)$;

(c). $|F^{m}_i(y)|\geq 2\tau$ for $y\in B(x_i, \frac{1}{3}\gamma_1r_x)\backslash B(x_i, \gamma_2r_x)$.

\medskip

When the value of $m$ increases,
we always assume $F^m_i$ converges to $F^m$ on $B(x, \frac{1}{2}\gamma_1r_x)$, after taking subsequences.
Furthermore, conditions (a)-(c) still hold. We further require that

(d). $|g^s_i|\leq \frac{\tau}{10^s}$ on $B(x_i, \gamma_1r_x)$ for $s>N_1$.

This can be achieved if we rescale functions $g^s_i$ for $s>N_1$.

Let $\Omega'_{im}$ be the connected component of $(F^{m}_i)^{-1}(B_{\mathbb{C}^m}(0, 1))$ containing $B(x_i, \gamma_2r_x)$. According to (b), $\Omega'_{im}\subset\subset B(x_i, \frac{1}{3}\gamma_1r_x)$. Then $F^{m}_i$ is a proper holomorphic map from $\Omega'_{im}$ to $B_{\mathbb{C}^m}(0, 1)$. By the proper mapping theorem, the image $W^m_i\ni 0$ is an irreducible analytic set in $B_{\mathbb{C}^m}(0, 1)$. We claim that $W^m_i$ has complex dimension $n$. Indeed, if this is not true, pick a generic point $z\in W^m_i$ with $|z|<\tau$. Then $(F^m_i)^{-1}(z)$ has dimension greater than $0$. By (c), $(F^m_i)^{-1}(z)$ is a compact analytic set in $B(x_i, \gamma_2r_x)$. Note $B(x_i, \gamma_2r_x)$ is contained in the Stein manifold $\Omega_i$ defined right below (\ref{76}). Thus, $(F^m_i)^{-1}(z)$ consists of finitely many points. Contradiction.

By (a),  there is a uniform gradient bound of $g^s_i$ on $B(x_i, \frac{1}{3}\gamma_1r_x)$. Then the image of $F^{m}_i(B(x_i, \frac{1}{3}\gamma_1r_x))$ has uniform volume upper bound.
Since $g^s_i$ is convergent for each $s$, $W^m_i$ is convergent in the Hausdorff metric sense to some $W^m$ in $B_{\mathbb{C}^m}(0, 1)$.
By a theorem of Bishop \cite{[B]}, $W^m$ is an analytic set of dimension $n$. By $(b)$, we find $F^m(B(x, \gamma_2r_x))\subset W^m$.

We claim that after adding finitely many functions, $(F^m)^{-1}(z)$ is unique for generic $z\in B_{\mathbb{C}^m}(0, \tau)\cap W^m$.  Note $W^m\cap B_{\mathbb{C}^m}(0, \tau)$ has finitely many irreducible components, say $W^{m1}, ... , W^{mj}$. Let $\Sigma'_m= F^m(B(x, \gamma_2r_x)\backslash G)$. Then $\Sigma'_m$ has codimension at least $4$ in $W^m$, as $F^m$ is Lipschitz. Therefore the regular points of $W^{mh}\backslash\Sigma'_m(1\leq h\leq j)$ are connected. 
According to (c), the preimage of any point in $(W^{mh}\backslash\Sigma'_m)\cap B_{\mathbb{C}^m}(0, \tau)$ is a compact analytic subvariety in $G\cap B(x, \gamma_2r_x)$. Thus we can separate it by adding only finitely many functions. We do this for all $1\leq h\leq j$. Then $(F^m)^{-1}(z)$ is unique for generic $z\in B_{\mathbb{C}^m}(0, \tau)\cap W^m$. Say now $m = N_2$.

Next we prove that for some larger $m$, a small neighborhood of $x$ is homeomorphic to $W^m\cap B_{\mathbb{C}^m}(0, \frac{\tau}{3})$. For any $k>l\geq N_2$, there exists a natural projection $P_{kl}$: $W^k\to W^l$.
We have $P_{kl}\circ F^k = F^l$ on $B(x, \gamma_2r_x)$. Let $z\in W^l\cap B_{\mathbb{C}^l}(0, \frac{1}{3}\tau)$. Then by (d), $P_{kl}^{-1}(z)$ is a compact analytic subvariety in $W^k\cap B_{\mathbb{C}^k}(0, \frac{\tau}{2})$. Hence it contains only finitely many points. Similar as on page $90$ of \cite{[DS]},  the number of $P_{kl}^{-1}(z)$ is actually bounded by the number of locally irreducible component of $z$ in $W^l$. 
As on page $90$ of \cite{[DS]}, we may write $W^m\cap B_{\mathbb{C}^m}(0, \frac{\tau}{3})$ as a finite union of sets $Z_\alpha$ which are given by analytic variety minus analytic subvariety so that 
$(F^m)^{-1}(Z_\alpha)$ is a disjoint union of $n_\alpha$ copies of $Z_\alpha$.
By induction argument as on page $90$ of \cite{[DS]}, after adding finitely many functions, we find the preimage of $W^m\cap B_{\mathbb{C}^m}(0, \frac{\tau}{3})$ is unique. This proves that a small neighborhood of $x$ is homeomorphic to $W^m\cap B_{\mathbb{C}^m}(0, \frac{\tau}{3})$. Say now $m=N_3$.

Next we prove that $W^m\cap B_{\mathbb{C}^m}(0, \frac{\tau}{3})$ is locally irreducible for $m\geq N_3$. If this is not true, we can find $z\subset W^m\cap B_{\mathbb{C}^m}(0, \frac{\tau}{3})$ and $\lambda>0$ with $B_{\mathbb{C}^m}(z, \lambda)\cap W^m\subset B_{\mathbb{C}^m}(0, \frac{\tau}{3})$ and $B_{\mathbb{C}^m}(z, \lambda)\cap W^m$ is connected. Moreover, there exist holomorphic functions $u, v$ on $B_{\mathbb{C}^m}(z, \lambda)\cap W^m$ with $uv = 0$, but $u, v$ are not identically zero. Now $(F^m)^{-1}(B_{\mathbb{C}^m}(z, \lambda)\cap W^m)$ is a connected open set in $B(x, \gamma_2r_x)$. It is clear that $u, v$ are holomorphic on $G'=G\cap (F^m)^{-1}(B_{\mathbb{C}^m}(z, \lambda)\cap W^m)$.
Recall $\mathcal{R}_\epsilon$ in definition \ref{def1}. According to corollary \ref{cor-2}, if $\epsilon = \epsilon(n)$ is sufficiently small, $\mathcal{R}_\epsilon$ is regular in the holomorphic sense. That is, for any $y\in \mathcal{R}_\epsilon$, there exists a holomorphic chart around $y$. Note $\mathcal{R}$ is dense. Assume at $y\in \mathcal{R}\cap G'$, $u(y)\neq 0$. Then $v$ vanishes in a small neighborhood of $y$. By applying theorem $3.9$ in \cite{[CC3]} and the unique continuation of holomorphic functions, we find $v\equiv 0$. Contradiction.

Let $S_m$ be the singular set of $W^m\cap B_{\mathbb{C}^m}(0, \frac{\tau}{3})$. 
We claim that for some larger $m$,  $(F^m)^{-1}(S_m)\subset B(x, \gamma_2r_x)\backslash G$. This is equivalent to saying that $F^m$ maps $G\cap (F^m)^{-1}(B_{\mathbb{C}^m}(0, \frac{\tau}{3}))$ to the regular part of $W^m$. Note this is also equivalent to that the holomorphic structure on $G\cap (F^m)^{-1}(B_{\mathbb{C}^m}(0, \frac{\tau}{3}))$ is the same as the one induced from  $W^m\cap B_{\mathbb{C}^m}(0, \frac{\tau}{3})$.
$S_m$ is a finite union of irreducible analytic sets in $B_{\mathbb{C}^m}(0, \frac{\tau}{3})$. Let $S_m^t(1\leq t\leq l)$ be irreducible components so that $(F^m)^{-1}(S_m^t)$ intersects $G$. Pick a point $y\in (F^m)^{-1}(S_m^t)\cap G$. According to proposition \ref{prop-10}, we can find sequences of holomorphic functions $\lambda^i_j$ on $B(x_i, \gamma_1r_x)$. Also $\lambda^i_j\to \lambda^\infty_j$ and $\lambda^\infty_j$ form a holomorphic coordinate near $y$. If we add these functions to $g^s_i$(with certain normalizations), the dimension of $S_m^t$ decreases. Then the claim follows from a standard induction.

For $x\in M_\infty$ as above, we consider the analytic structure in a neighborhood induced by $F^m$. Let $\mathcal{O}$ be the structure sheaf and $\mathcal{O}_x$ be the stalk at $x$.
Now we prove that after adding finitely many functions, $\mathcal{O}_x$ is normal. 
 There exists an open set $(F^m)^{-1}(B_{\mathbb{C}^m}(0, \frac{\tau}{3}))\supset U\ni x$ and a normalization $\hat{U}\to U\ni x$ so that $\mathcal{O}(\hat{U})$ is a finite module over $\mathcal{O}(U)$. Note by $(14.11)$ on page $89$ of  \cite{[GPR]}, the natural map $\hat{U}\to U$ is a homeomorphism, as $U$ is locally irreducible. Let us assume $\mathcal{O}(\hat{U})$ is generated by $u_1, ..., u_k\in \Gamma(U, \mathcal{F})$ over $\mathcal{O}(U)$. Thus they extend to continuous functions on $U$. According to lemma \ref{lm6}, there exist $\delta>0, \epsilon_0>0$ and holomorphic functions $u^i_j$ on $B(x_i, 2\delta)$ with $u^i_j\to u_j(1\leq j\leq k)$ uniformly on $B(x, \epsilon_0)\cap G$. By adding these functions to $g^s_i$ and shrinking the neighborhood of $x$, we find that $\mathcal{O}_x$ is normal. 
 
 As normal points are open (theorem $14.4$ on page $87$ of \cite{[GPR]}), we proved that for any point $x\in M_\infty$, there exists a neighborhood $U_x\ni x$ so that $U_x$ is a normal analytic variety with structure sheaf $\mathcal{O}(x)$. Let $z\in U_x\cap U_y$. To prove that $M_\infty$ is a normal complex analytic variety, it suffices to prove that $(\mathcal{O}(x))_z = (\mathcal{O}(y))_z$(stalk) for $z\in V\subset U_x\cap U_y$. Let $f\in \Gamma(V, \mathcal{O}(x))$.
Then $f|_{V\cap G}\in \Gamma(G\cap V, \mathcal{O}(y))$. As $V\backslash G$ has real codimension $4$ and $\mathcal{O}(y)$ is normal, $f\in\Gamma(V, \mathcal{O}(y))$.  
This completes the proof of theorem \ref{thm1}.

\end{document}